\documentclass[10pt]{amsart}
\usepackage{amsmath,amsthm,amscd,amsfonts,amssymb,graphicx,paralist,color}
\usepackage{enumitem}
\usepackage[initials]{amsrefs}

\usepackage{xcolor}
\usepackage{hyperref}
\usepackage{rotating}
\hypersetup{colorlinks=false,linkbordercolor=red,linkcolor=green,pdfborderstyle={/S/U/W 1}}
\usepackage[skip=0.7\baselineskip]{caption}
\usepackage[utf8]{inputenc}
\usepackage[T1]{fontenc}

\usepackage{tikz}
\usetikzlibrary{trees}

\usepackage{enumerate}

\theoremstyle{plain} 
\newtheorem{theorem}{Theorem}[section]
\newtheorem{lemma}[theorem]{Lemma} 
\newtheorem{corollary}[theorem]{Corollary} 
\newtheorem{proposition}[theorem]{Proposition}

\newtheorem{definition}[theorem]{Definition}

\newtheorem{remark}[theorem]{Remark}

\newcommand{\charf}[1]{\mbox{\raise.48ex\hbox{$\chi$}$_{#1}$}}

\newcommand{\N}{\mathcal{N}}

\newcommand{\R}{\mathbb{R}}

\newcommand{\A}{\mathcal{A}}

\newcommand{\ad}{\mathrm{add}}

\newcommand{\co}{\mathfrak{c}}

\newcommand{\la}{\langle}
\newcommand{\ra}{\rangle}

\DeclareMathOperator{\card}{card}
\DeclareMathOperator{\cov}{cov}
\DeclareMathOperator{\cof}{cof}
\DeclareMathOperator{\ZFC}{ZFC}
\DeclareMathOperator{\non}{non}
\DeclareMathOperator{\add}{add}

\makeatletter
\@namedef{subjclassname@1991}{2020 Mathematics Subject Classification}
\makeatother

\title[Lineability on nets and uncountable sequences of functions]
{Lineability on nets and uncountable sequences of functions in measure theory}
\author[Rodr\'{i}guez-Vidanes]{Daniel L.~Rodr\'{i}guez-Vidanes}
\address[Daniel L.~Rodr\'{i}guez-Vidanes]{\mbox{}\newline\indent Instituto de Matem\'atica Interdisciplinar (IMI), \newline \indent Departamento de An\'{a}lisis y Matem\'{a}tica Aplicada \newline \indent Facultad de Ciencias Matem\'{a}ticas \newline \indent Plaza de Ciencias 3 \newline \indent Universidad Complutense de Madrid \newline \indent	Madrid, 28040 (Spain).}
\email{dl.rodriguez.vidanes@ucm.es}

\begin{document}

\subjclass{15A03; 28A20; 54A20; 54A25; 03E35}

\keywords{Dominated Convergence Theorem, Monotone Convergence Theorem, Fatou's Lemma, measurable function, nets, uncountable sequences, algebrability}
\thanks{The author was supported by Grant PGC2018-097286-B-I00 and by the Spanish Ministry of Science, Innovation and Universities and the European Social Fund through a “Contrato Predoctoral para la Formación de Doctores, 2019” (PRE2019-089135).
The paper was partially done at West Virginia University (WVU) during a stay of the author, who thanks the support provided by the Department of Mathematics at WVU.
The author also acknowledges the help of his PhD advisors, Prof. Krzysztof C. Ciesielski, Prof. Juan B. Seoane-Sepúlveda and Prof. Gustavo A. Muñoz-Fernández, whose insightful comments improved the presentation and results of this paper}

\begin{abstract}
	In general, some of the well known results of measure theory dealing with the convergence of sequences of functions such as the Dominated Convergence Theorem or the Monotone Convergence Theorem are not true when we consider arbitrary nets of functions instead of sequences.
	In this paper, we study the algebraic genericity of families of nets of functions that do not satisfy important results of measure theory, and we also analyze the particular case of uncountable sequences.
\end{abstract}

\maketitle

%\tableofcontents
%
%\newpage
\section{Introduction and preliminaries}

The Dominated Convergence Theorem (DCT), Fatou's Lemma (FL), the Monotone Convergence Theorem (MCT), the pointwise limit of a sequence of measurable functions is measurable (LM), and almost everywhere (a.e.) convergence implying convergence in measure in finite measure spaces (AE) are cornerstone results of Measure Theory (see \cite{Fo}).
They are applied over sequences of functions indexed by $\mathbb N$ endowed with the usual ordering, i.e., sequences of the form $(f_n)_{n\in \mathbb N}$.
However, the DCT, FL, the MCT, LM and AE are not true, in general, for arbitrary nets of functions (see \cite[pp.~415]{AlBo}).
The main goal of this paper is to show that there are ``large'' algebraic structures such as algebras and cones contained in the families of nets and uncountable sequences of functions that do not satisfy the DCT, the MCT, FL, LM or AE.
This search of algebraic structures is known as lineability.

It is important to mention that the DCT is satisfied for arbitrary nets of functions when the measure space is countable (\cite[lemma~2.3]{BeDe}).
Furthermore, in \cite{Ma3}, the author proves that the DCT, FL and the MCT are also true in the case of the Lebesgue measure over $\mathbb R$ (or $\mathbb C$) for nets of functions indexed by \textit{countably accessible directed sets} (concept introduced by 
J.E. Marsden in \cite{Ma2}).

Let us now present the necessary background of lineability which will be the major line of study throughout the paper.
Lineability has been a fruitful field of research for many mathematicians since the term was coined by V.I. Gurariy in the early 2000's (see \cites{ArGuSe,Se}).
The aim of lineability is the search of algebraic structures within subsets of linear spaces that are not necessarily linear.
Although the term lineability was introduced by Gurariy and added recently to the AMS subject classification under 15A03 and 46B87, the result that may have started this field is perhaps that of Levin and Milman in 1940 (\cite{LeMi}), which states that the set of all functions of bounded variation on $[0,1]$ does not contain a closed infinite dimensional subspace in $C([0,1])$
endowed with the supremum norm.
%\ch{This is one of the negative results we were searching for. Add this to our list.\\}

Since its introduction, lineability has made an impact in various areas of mathematics such as Set Theory (\cite{CiSe}), Measure Theory (\cites{MuPaPu,BoCaFaPeSe,BeCa}), Complex Analysis (\cite{BeCaSe}), non-Archimedean Analysis (\cite{FeMaRoSe}), linear dynamics (\cite{ArCoPeSe}), among others.
For more information and a current state of the art on this topic, we also refer the interested reader to \cites{ArBeMuPrSe,ArBePeSe,BaBiFiGl,BeCaMuSe,BeFeMaSe,CiRoSe} and the references therein.

The following terminology has been used in several results of lineability (see \cites{AiPeGaSe,FeGeTr,AdFaRoSe}).
Let $V$ be a vector space (over $\mathbb R$ or $\mathbb C$), $A\subset V$ and $\kappa$ a cardinal number.
\begin{itemize}
	\item $A$ is \textbf{positively (resp. negatively) $\kappa$-coneable} if there exists a $B\subset A$ with $\text{card}(B)=\kappa$ of linearly independent vectors in $V$ such that the positive (resp. negative) cone generated by $B$ is contained in $A$, i.e.,
		$$
		\bigcup_{n\in \mathbb N} \left\{\sum_{i=1}^n a_i x_i \colon a_1,\ldots,a_n>0\ (\text{resp. }<0) \text{ and } x_1,\ldots,x_n\in B \right\}\subset A.
		$$
	If $A$ is positively and negatively $\kappa$-coneable, then we simply say that $A$ is \textbf{$\kappa$-coneable}.
\end{itemize}
It is important to mention that, although the definition of positively (negatively) $\kappa$-coneable is applied correctly in the references above, the way $\kappa$-coneable 
is given in all them is not entirely correct.
Here we present the precise definition of positively (negatively) $\kappa$-coneable.

The main results of this paper consider the following terminology which is by now standard (see \cite{ArBePeSe}).
Let $V$ be a vector space contained in a (not necessarily unital) algebra (over a field $\mathbb K$), $A\subset X$, and $\kappa$ and $\mu$ be cardinal numbers.
\begin{itemize}
	\item $A$ is \textbf{$(\kappa,\mu)$-algebrable} if there exists an algebra $B$ such that $B\setminus \{0\}\subset A$, the dimension of $B$ (as a vector space) is equal to $\kappa$ and there is a minimal system of generators $S$ for the algebra $B$ such that $\card(S)=\mu$.
	If $\kappa=\mu$ we simply say that $A$ is \textbf{$\kappa$-algebrable}.
\end{itemize}
By $S$ being a minimal system of generators for $B$ we mean that the algebra $B$ is generated by $S$ and, for every $s\in S$, $s$ does not belong to the algebra generated by $S\setminus \{s\}$.

\begin{itemize}
	\item $A$ is \textbf{strongly $\kappa$-algebrable} if there exists a $\kappa$-generated free algebra $B$ such that $B\setminus \{0\}\subset A$.
\end{itemize}

Recall that an algebra $B$ is a $\kappa$-generated free algebra (over a field $\mathbb K$) if there exists $X\subset B$ of cardinality $\kappa$ such that any function $f$ from $X$ to some algebra $\mathfrak A$ (over $\mathbb K$) can be uniquely extended to a homomorphism from $B$ to $\mathfrak A$.
We call $X$ the set of \textit{free generators} of the algebra $B$.
In a commutative algebra there is a simpler equivalent definition of a $\kappa$-generated free algebra.
In particular, $X=\{x_\alpha\colon \alpha<\kappa \}$ contained in a commutative algebra $\mathfrak A$ generates a free subalgebra $B$ if, and only if, (1) the algebra generated by $X$ is $B$ and (2) for any polynomial $P$ of degree $n\in \mathbb N$ and without free term and for any
$x_{\alpha_1},\ldots,x_{\alpha_n}\in X$ distinct we have that $P(x_{\alpha_1},\ldots,x_{\alpha_n}) = 0$ if and only if $P = 0$.
If $X\subset B$ satisfies (2), we say that the vectors in $X$ are \textit{algebraically independent}.
Note that in this case $X$ is a set of free generators if and only if the set of all elements $x_{\alpha_1}^{k_1} x_{\alpha_2}^{k_2}\cdots x_{\alpha_n}^{k_n}$, where $k_i$ belongs to $\mathbb N$ for every $i\in \{1,\ldots,n \}$, is linearly independent and generates $B$.

\subsection{Basic concepts}

Throughout the paper we will use standard notation and definitions from set theory. 
Ordinal numbers will be identified with the set of their predecessors and cardinal numbers with the initial ordinals. 
Given a set $A$, the cardinality of $A$ will be denoted by $\card(A)$.
The set of functions from $A$ to $B$ will be denoted by $B^A$.
The sets of positive integers, non-negative integers, real numbers and the closed unit interval will be denoted by $\mathbb N$, $\mathbb N_0$, $\mathbb R$ and $[0,1]$, respectively.
For any $X\subseteq [0,1]$ we will denote by $\chi_X$ the characteristic function of $X$.
%The standard Cantor ternary set will be denoted by $\mathfrak C$.
Let $\omega$, $\omega_1$, $\omega_2$, and $\mathfrak c$
 denote the first infinite cardinal, the first uncountable cardinal, the second uncountable cardinal and the cardinality of $\R$%$\mathfrak c$
 , respectively.

Given a measure space $(\Omega,\mathcal F,\mu)$, we say that a real-valued measurable function $f$ defined on $\Omega$ 
is integrable if $\int_\Omega |f| d\mu<\infty$.
We will denote by $\lambda$ the Lebesgue measure on the real line.
Also let $\mathcal N$ denote the family of all Lebesgue measure zero (null) sets contained in $[0,1]$.
By $\text{non}(\mathcal N)$, $\text{cov}(\mathcal N)$ and $\text{add}(\mathcal N)$ we denote, respectively, the \textit{uniformity number}, \textit{covering number} and \textit{additivity number} of $\mathcal N$ which are defined as:
\begin{align*}
\text{non}(\mathcal N) & :=\min \{\card(X) \colon X\subseteq [0,1] \text{ and } X\notin \mathcal N \},\\
\text{cov}(\mathcal N) & :=\min \{\card(X) \colon X\subset \mathcal N \text{ and } \bigcup X=[0,1] \},\\
\text{add}(\mathcal N) & :=\min \{\card(X) \colon X\subset \mathcal N \text{ and } \bigcup X \notin \mathcal N \}.
\end{align*}
It is well known that 
	$$
	\omega_1\leq \text{add}(\mathcal N)\leq \text{non}(\mathcal N),\text{cov}(\mathcal N)\leq \mathfrak c.
	$$
For more information on $\text{non}(\mathcal N)$, $\text{cov}(\mathcal N)$ and $\text{add}(\mathcal N)$, see \cite{BaJu}.

\subsection{Nets and uncountable sequences}

The main results of this paper deal with nets and uncountable sequences of functions.
For completeness, let us recall firstly some standard notations and definitions of nets.
A directed set is a nonempty set $\mathcal A$ endowed with a preorder $\leq$ (that is, a reflexive and transitive binary relation) such that every pair of elements of $\mathcal A$ has an upper bound.
A net in a set $X$ is a function $x\colon \mathcal A\to X$, where $\mathcal A$ is a directed set.
The directed set $\mathcal A$ is called the index set and the elements of $\mathcal A$ are called indexes.
We will denote the function $x(\cdot)$ by $(x_A)_{A\in \mathcal A}$.
Notice that $\mathbb N$ with the usual ordering is a directed set.

As in the case of sequences, we can define the convergence of a net contained in a topological space.
A net $(x_A)_{A\in \mathcal A}$ in a topological space $(X,\tau)$ converges to $x\in X$ if $(x_A)_{A\in \mathcal A}$ is eventually in every neighborhood of $x$, i.e., for each neighborhood $U^x$ of $x$ there exists $A_0\in \mathcal A$ such that $x_A\in U^x$ for every $A\geq A_0$.
We say that $x$ is the \textit{limit} of $(x_A)_{A\in \mathcal A}$ and write $x=\lim_{A\in \mathcal A} x_A$.
%If $(x_A)_{A\in \mathcal A}\subset \R$, we say that $(x_A)_{A\in \mathcal A}$ diverges to $\infty$ (resp. $-\infty$) if for every $n\in \mathbb N$, there exists $A_0\in \mathcal A$ such that $x_A\geq n$ (resp. $x_A\leq -n$) for every $A\geq A_0$.
%We denote $\lim_{A\in \mathcal A} x_A=\pm \infty$ when $(x_A)_{A\in \mathcal A}$ diverges to $\pm \infty$.

For completeness, we will show the following simple fact on the limits of nets of real numbers.

\begin{remark}\label{rem:3}
	Let $\mathcal A$ be a directed set. 
	If $(x_A)_{A\in \mathcal A}$ and $(y_A)_{A\in \mathcal A}$ are nets of real numbers that converge, respectively, to $x$ and $y$, and $\alpha,\beta\in \R$, then the net $(\alpha x_A+\beta y_A)_{A\in \mathcal A}$ converges to $\alpha x+\beta y$.
	Indeed, it is straightforward if $\alpha$ or $\beta$ are $0$.
	So assume that $\alpha,\beta\neq 0$.
	Fix $\varepsilon>0$.
	By definition of convergence, there exist $A_0^x,A_0^y\in \mathcal A$ such that $|x_A-x|<\frac{\varepsilon}{2 |\alpha|}$ for every $A\geq A_0^x$ and $|y_A-y|<\frac{\varepsilon}{2 |\beta|}$ for every $A\geq A_0^y$.
	Since $A$ is a directed set, there is an $A_0\in \mathcal A$ such that $A\geq A_0^x,A_0^y$.
	Hence, for every $A\geq A_0\geq A_0^x,A_0^y$ we have, by the triangular inequality, that 
		$$
		|\alpha x_A+\beta y_A-(\alpha x+\beta y)|\leq |\alpha||x_A-x|+|\beta||y_A-y|<\varepsilon.
		$$
	Therefore, the limit operator $\lim_{A\in \mathcal A}(\cdot)$ on converging nets of real numbers is linear.
\end{remark}

An element $x$ in a topological space $(X,\tau)$ is called a \textit{limit point} (or a cluster point) of a net $(x_A)_{A\in \mathcal A}$ if for each neighborhood $U^x$ of $x$ and each index $A\in \mathcal A$ there exists $A_0\in \mathcal A$ with $A_0\geq A$ such that $x_{A_0}\in U^x$.
We will denote by $\text{Lim}(x_A)$ the (possibly empty) set of limit points of $(x_A)_{A\in \mathcal A}$.
If $(x_A)_{A\in \mathcal A}$ is a bounded net of real numbers then, as in the case of sequences, there exists the infimum and supremum of $\text{Lim}(x_A)$ called the \textit{limit inferior} and \textit{limit superior} of $(x_A)_{A\in \mathcal A}$, respectively, and they are also finite.
We will denote the limit inferior and limit superior of a net of real numbers $(x_A)_{A\in \mathcal A}$ by $\liminf_{A\in \mathcal A} x_A$ and $\limsup_{A\in \mathcal A} x_A$, respectively.
It is easy to see that 
$$
\liminf_{A\in \mathcal A} x_A\leq\limsup_{A\in \mathcal A} x_A
$$ 
and, furthemore, $\liminf_{A\in \mathcal A} x_A=\limsup_{A\in \mathcal A} x_A$ if and only if $(x_A)_{A\in \mathcal A}$ converges to $\liminf_{A\in \mathcal A} x_A=\limsup_{A\in \mathcal A} x_A$.
For more information on nets, we refer the interested reader to \cite[section~2.4]{AlBo}.

Let us approach now the case of uncountable sequences.
Let $X$ be a nonempty set.
A transfinite sequence or an $\omega_1$-sequence in $X$ is a function $x\colon \omega_1\to X$, denoted by $(x_\alpha)_{\alpha<\omega_1}$.
Now, if we assume that $X$ is a topological space, an $\omega_1$-sequence $(x_\alpha)_{\alpha<\omega_1}$ converges to $x\in X$ if for every neighborhood $U^x\subset X$ of $x$,  there is an ordinal $\alpha_0<\omega_1$ such that $x_\alpha \in U^x$ for any $\alpha>\alpha_0$ with $\alpha<\omega_1$.
In the literature $x$ is known as the transfinite limit or \textit{$\omega_1$-limit} of $(x_\alpha)_{\alpha<\aleph_1}$ and it is denoted by $x=\lim_{\alpha<\omega_1} x_\alpha$.
(The notions of $\omega_1$-sequences and $\omega_1$-limits can be traced back to a 1907 paper of J. Mollerup \cite{Mo}.)
The main results of this paper dealing with transfinite sequences will be considered in the real line with the Euclidean metric.
Notice that in a metric space every convergent $\omega_1$-sequence is eventually constant as any metric space is first countable.
This is shown in Remark~\ref{rem:1} below.

\begin{remark}\label{rem:1}
	Let $X$ be a first countable space.
	A transfinite sequence $(x_\alpha)_{\alpha<\omega_1}$ is convergent if, and only if, there is an ordinal $\alpha_0<\omega_1$ such that $x_\alpha=x_{\alpha_0}$ for every $\alpha> \alpha_0$ with $\alpha<\omega_1$.
%	\ck{This is true for any FIRST countable space.\\} 
\end{remark}

In 1920, W. Sierpi\'{n}ski in \cite{Si} began the study of convergence of transfinite sequences of real functions $(f_\alpha)_{\alpha<\omega_1}$. 
An $\omega_1$-sequence of functions $(f_\alpha)_{\alpha<\omega_1}\subset X$ is said to converge pointwise to a function $f\in {\mathbb R}^X$ (called the $\omega_1$-limit of $(f_\alpha)_{\alpha<\omega_1}$) if $f(x)=\lim_{\alpha<\omega_1} f_{\alpha}(x)$ for every $x\in X$.
For a family $\mathcal F\subset {\mathbb R}^X$, we define the transfinite closure or $\omega_1$-closure of $\mathcal F$  as the 
family:
	$$
	\text{LIM}_{\omega_1} (\mathcal F)=\left\{\lim_{\alpha<\omega_1} f_\alpha \colon (f_\alpha)_{\alpha<\omega_1}\subset \mathcal F \right\}.
	$$
We say that $\mathcal F$ is closed with respect to transfinite limits or $\omega_1$-closed if $\text{LIM}_{\omega_1}(\mathcal F)=\mathcal F$.
If $\text{LIM}_{\omega_1}(\mathcal F)=\mathbb R^X$, then we say that $\mathcal F$ is $\omega_1$-dense.
In \cite{Si}, Sierpi\'{n}ski proved that the family of continuous real functions, $C(\mathbb R)$, and the family of Baire 1 functions, $\mathcal B_1(\mathbb R)$, are $\omega_1$-closed, i.e., the pointwise limit of an $\omega_1$-sequence of continuous real functions is continuous (analogously with the class $\mathcal B_1(\mathbb R)$). 
Moreover, Sierpi\'{n}ski proved, in particular, that if an $\omega_1$-sequence of continuous functions converges, then the $\omega_1$-sequence is eventually constant (the idea is that every continuous function is determined on the rational numbers, i.e., a countable set).
Interestingly this is not the case for Baire $1$ functions as the following example provided by Sierpi\'{n}ski shows.
Let $\{x_\alpha \colon \alpha<\omega_1  \}$ be a subset of $\R$ of cardinality $\omega_1$, then, for every $\alpha<\omega_1$, let
	$$
	f_\alpha(x)=\chi_{\{x_\alpha\}}.
	$$
Notice that $(f_\alpha)_{\alpha<\omega_1}$ is an $\omega_1$-sequence of distinct Baire $1$ functions that converges pointwise to $0$.
(Compare the construction of the transfinite sequence $(f_\alpha)_{\alpha<\omega_1}$ with the net $(f_A)_{A\in \A}$ in Remark~\ref{rem:2} (3).)
In the case of the class of Baire 2 functions, $\mathcal B_2(\mathbb R)$, it is known that it is independent of ZFC (where ZFC stands for the Zermelo-Fraenkel axiomatic system with the axiom of choice) that $\mathcal B_2(\mathbb R)$ is $\omega_1$-closed 
(see \cite{Ko} or \cite{Na1}).
In fact, under ZFC with the Continuum Hypothesis (CH), Sierpi\'{n}ski also proved that $\mathcal B_2(\mathbb R)$ is $\omega_1$-dense.

Throughout the 20th century the $\omega_1$-closure of several classes of families of real functions has been studied by many mathematicians (see, for instance, \cites{Li,Ne} for $\omega_1$-closed families and \cites{Gr,Na4} for $\omega_1$-dense families).
The concepts of convergence and the definition of $\omega_1$-sequences can be generalized analogously as $\alpha$-sequences for any ordinal $\alpha$; though the convergence of $\alpha$-sequences is reducible to convergence of $\cof(\alpha)$-sequences.
For completeness, we will study why such reduction is considered in Proposisitions~\ref{prop:1} and~\ref{prop:2}.

In this paper, we will be primarily interested in studying the lineability of nets of functions that do not satisfy the DCT, the MCT, FT, AE and LM, but we will pay special attention to the particular case of $\kappa$-sequences.

In \cite{Ko}, P. Komj\'{a}th proved that it is consistent with $\neg$CH that $\mathcal B_2(\mathbb R)$ is $\omega_2$-closed and it is also consistent with $\neg$CH that $\mathcal B_2(\mathbb R)$ is $\omega_2$-dense.
The $\mathfrak c$-closure of arbitrary families of functions has also been studied when $\mathfrak c$ is a regular cardinal (see \cite[theorem~2.11]{Na5}).
In this paper, 
we will be able to 
examine %study 
the $\add(\N)$-closure and the $\non(\N)$-closure of the family of Lebesgue measurable real 
functions, denoted by $\mathcal L$, as
a consequence of the main results.
The $\omega_1$-closure of $\mathcal L$ has been studied by many mathematicians under several set-theoretical assumptions, see \cites{Di,NaWe,NoSzWe}.

Finally, let show that it makes sense to study the existence of algebraic structures inside families of nets.
Given a directed set $\mathcal A$ and $X\subset \R$ a nonempty set, we will %be 
denote
by $\left( \mathbb{R}^X\right)^{\mathcal A}$ the family of nets of real-valued functions on $X$ indexed by $\mathcal A$ endowed with the following operations of addition, product and scalar multiplication over $\mathbb R$, respectively.
	\begin{align*}
		& +\colon \left( \mathbb{R}^X\right)^{\mathcal A}\times \left(\mathbb{R}^X\right)^{\mathcal A} \to \left( \mathbb{R}^X\right)^{\mathcal A}\\
		& \quad \quad \left((f_A)_{A\in \mathcal A}, (g_A)_{A\in \mathcal A}\right) \mapsto (f_A)_{A\in \mathcal A}+(g_A)_{A\in \mathcal A}:=(f_A+g_A)_{A\in \mathcal A},\\
		& \bullet\colon \left( \mathbb{R}^X\right)^{\mathcal A}\times \left( \mathbb{R}^X\right)^{\mathcal A} \to \left( \mathbb{R}^X\right)^{\mathcal A}\\
		& \quad \quad\left((f_A)_{A\in \mathcal A}, (g_A)_{A\in \mathcal A}\right) \mapsto (f_A)_{A\in \mathcal A}\bullet (g_A)_{A\in \mathcal A}:=(f_A g_A)_{A\in \mathcal A},\\
		& \cdot \colon \mathbb{R}\times \left( \mathbb{R}^X\right)^{\mathcal A} \to \left( \mathbb{R}^X\right)^{\mathcal A}\\
		& \quad \ \left(r, (f_A)_{A\in \mathcal A}\right) \mapsto r\cdot (f_A)_{A\in \mathcal A}:=(r f_A)_{A\in \mathcal A}.
	\end{align*}
It is easy to see that $\left( \mathbb{R}^X\right)^{\mathcal A}$ endowed with $+$, $\bullet$ and $\cdot$ is a commutative algebra over $\mathbb R$.
For every $n\in \mathbb N_0$ and $(f_A)_{A\in \mathcal A}\in \left( \mathbb{R}^X\right)^{\mathcal A}$, we denote by $(f_A)_{A\in \mathcal A}^{n}$ the product of $(f_A)_{A\in \mathcal A}$ $n$-times if $n\geq 1$, and $(\chi_{X})_{A\in \mathcal A}$ if $n=0$.

Analogously, we will denote $\left( \mathbb{R}^X\right)^{\kappa}$ the algebra of $\kappa$-sequences of real-valued functions on $X$, where $\kappa$ is an infinite initial ordinal ($\kappa=\omega$ is simply the classical case of countable sequences $(f_n)_{n\in \mathbb N}$).

It is well known that if we have a vector space $V$ over a field $\mathbb K$ of infinite dimension, then $\card(V)=\max\{\dim(V),\card(\mathbb K) \}$, where $\dim(V)$ denotes the dimension of $V$.
By considering, for instance, the space of $\omega$-sequences $\left( \R^\R\right)^{\omega}$, we have that $\left( \R^\R\right)^{\omega}$ has infinite dimension.
In fact, since $\card\left(\left( \R^\R\right)^{\omega} \right)=2^{\mathfrak c}>\mathfrak c=\card(\R)$, we have that $\dim\left(\left( \R^\R\right)^{\omega} \right)=2^{\mathfrak c}$.
Moreover, for any cardinal number $\kappa\geq \omega$, note that $\left( \R^\R\right)^{\kappa}$ has infinite dimension and $\dim\left(\left( \R^\R\right)^{\kappa} \right)=2^{\kappa \cdot \mathfrak c}$.
Indeed, since $\card\left(\left( \R^\R\right)^{\kappa} \right)=2^{\kappa\cdot \mathfrak c}>\mathfrak c=\card(\R)$ we obtain that $\card\left(\left( \R^\R\right)^{\kappa} \right)=\dim\left(\left( \R^\R\right)^{\kappa} \right)$.
Therefore, by taking $\kappa\geq 2^{\mathfrak c}$, there always exists a vector space of $\kappa$-sequences of real functions exceeding the dimension of 
$\R^\R$.
For this reason, we will focus on trying to find the smallest possible directed set $\mathcal A$ (resp., cardinal numbers $\kappa$) in terms of cardinality such that the algebraic structures contained in our families of nets (resp., $\kappa$-sequences) has dimension at most $2^{\mathfrak c}$.

\subsection{``Monster'' families and summary of results}

Let us define the following families of nets.
Fix %$\mathcal A$ 
a directed set $\mathcal A$.
\begin{itemize}
	\item[$\mathcal{NM}_{\mathcal A}$:] The family of nets of Lebesgue measurable functions $(f_A)_{A\in \mathcal A}\in \left( \mathbb{R}^{[0,1]}\right)^{\mathcal A}$ such that: (1) if $A\leq B$, then $0\leq f_A\leq f_B$ a.e., (2) $(f_A)_{A\in \mathcal A}$ converges pointwise a.e. to an integrable function $f\in \mathbb{R}^{[0,1]}$, and (3) 
	 	$$
	 	\lim_{A\in \mathcal A} \int_{[0,1]} f_A d\lambda\neq \int_{[0,1]} f d\lambda.
	 	$$ 
	
	\item[$\mathcal{NF}_{\mathcal A}$:] The family of nets of Lebesgue measurable functions $(f_A)_{A\in \mathcal A}\in \left( \mathbb{R}^{[0,1]}\right)^{\mathcal A}$ with $f_A\geq 0$ a.e. such that the function $\liminf_{A\in \mathcal A}(f_A(x))$ is integrable and $\int_{[0,1]} \liminf_{A\in \mathcal A} f_A d\lambda > \liminf_{A\in \mathcal A} \int_{[0,1]} f_A d\lambda$.
	
	\item[$\mathcal{AN}_{\mathcal A}$:] The family of nets of Lebesgue measurable functions $(f_A)_{A\in \mathcal A}\in (\mathbb R^\R)^{\mathcal A}$ that converge a.e. to a nonmeasurable function $f\in \mathbb R^\R$.
\end{itemize}
	
\begin{remark}\label{rem:4}
	Observe that the families of nets that satisfy the Monotone Convergence Theorem and the ones that satisfy Fatou's Lemma clearly form a positive cone.
	It is also obvious that the families of nets of measurable functions that converge a.e. to a measurable function already form an algebra.
	Therefore it is not necessary to study the coneability or algebrability in these 
	cases.
\end{remark}

\begin{itemize}	
	\item[$\mathcal{ND}_{\mathcal A}$:] The family of nets of Lebesgue measurable functions $(f_A)_{A\in \mathcal A}\in \left( \mathbb{R}^{[0,1]}\right)^{\mathcal A}$ such that (1) $|f_A|\leq g$ a.e. and all $A\in \mathcal A$, for some integrable function $g\in \mathbb R^{[0,1]}$, (2) $\lim_{A\in \mathcal A} f_A(x)=f(x)$ a.e., for some integrable function $f\in \mathbb R^{[0,1]}$, and (3) $\lim_{A\in \mathcal A} \int_{[0,1]} |f_A-f| d\lambda \neq 0$.
	
	\item[$\mathcal{ANM}_{\mathcal A}$:] The family of nets of Lebesgue measurable functions $(f_A)_{A\in \mathcal A}\in \left( \mathbb{R}^{[0,1]}\right)^{\mathcal A}$ that converges %converge 
	a.e. to a measurable function $f\in \mathbb{R}^{[0,1]}$, but $(f_A)_{A\in \mathcal A}$ does not converge in measure to $f$.
\end{itemize}

\begin{remark}\label{rem:2}
	\begin{itemize}
		\item[(1)] It is clear that $\mathcal{AN}_\omega$, $\mathcal{NM}_\omega$, $\mathcal{NF}_\omega$, $\mathcal{ND}_\omega$ and $\mathcal{ANM}_\omega$ are empty sets.
		
		\item[(2)] We have that $\mathcal{ND_A}\subseteq \mathcal{ANM_A}$.
		This is an immediate consequence of the definition of $\mathcal{ND_A}$ since (2) and (3) in the definition of $\mathcal{ND_A}$ imply that $(f_A)_{A\in \A}$ converges a.e. to a measurable function $f\in \mathbb R^{[0,1]}$ but not in measure.

		\item[(3)] It is not true, in general, that:
		if $(f_A)_{A\in \mathcal A}\in \left(\R^{[0,1]} \right)^{\mathcal A}$ and $g\in \mathbb R^{[0,1]}$ satisfy that $|f_A|\leq g$ a.e. for all $A\in \A$ and $(f_A)_{A\in \mathcal A}$ converges a.e. to a function $f\in \R^{[0,1]}$, then $|f|\leq g$ a.e.
		Indeed, let $\mathcal A$ be the set of all finite subsets of $[0,1]$ endowed with the binary relation of containment $\subseteq$.
		For every $A\in \mathcal A$, define 
			$$
			f_A:=\chi_A.
			$$
		Since $A$ is finite, it is obvious that $A\in \N$, i.e., $f_A$ is measurable and $f_A=0$ a.e. 
		However, observe that $(f_A)_{A\in \mathcal A}$ converges pointwise everywhere to $\chi_{[0,1]}$ which is strictly greater than $0$ at every $x\in [0,1]$.
	\end{itemize}
\end{remark}

In \cite{CaFeFiSe}, J. Carmona-Tapia et al. started for the first time to study lineability properties of families of nets from several points of view; one of them being Measure Theory. 
In particular, they show (implicitly) in \cite[theorem~2.13]{CaFeFiSe} and \cite[theorem~2.15]{CaFeFiSe}, respectively, that there are directed sets $\mathcal A$ and $\mathcal A^\prime$ both having cardinality $\mathfrak c$ such that $\mathcal{AN}_{\mathcal A}$ is $\mathfrak c$-lineable and $\mathcal{ND}_{\mathcal A^\prime}$ is strongly $\mathfrak c$-algebrable.
The main goal of their study was to find directed sets $\mathcal A$ such that $\mathcal{AN}_{\mathcal A}$ and $\mathcal{ND}_{\mathcal A^\prime}$ (with the $0$ net) contain large algebraic structures, but the size in terms of cardinality of the index set was not taken into account.
In this paper, we will provide improvements of \cite[theorems~2.13 and~2.15]{CaFeFiSe}

The main results are divided into two sections.
In Section~\ref{sec:2}, under ZFC, we will search for uncountable directed sets $\mathcal A$ that are not cardinal numbers such that the families of nets $\mathcal{NM_A}$, $\mathcal{NF_A}$, $\mathcal{ND_A}$, $\mathcal{AN_A}$ and $\mathcal{ANM_A}\setminus \mathcal{ND_A}$ contain cones or algebras.
See Table~\ref{tab:1} for a summary of the main results.

\begin{table}[h!]
	\begin{tabular}{|c|c|c|}
		\hline
		& $\card(\A)=\cov(\N)$ & $\card(\A)=\add(\N)$ \\ \hline
		Positively $2^{\mathfrak c}$-coneable & $\mathcal{NM_A}$, $\mathcal{NF_A}$, $\mathcal{ND_A}$ (Thm.~\ref{thm:1}) &  \\ \hline
		$2^{\mathfrak c}$-coneable & $\mathcal{ANM_A}\setminus \mathcal{ND_A}$ (Thm.~\ref{thm:7}) &  \\ \hline
		Strongly $\mathfrak c$-algebrable & \begin{tabular}[c]{@{}c@{}} $\mathcal{ANM_A}\setminus \mathcal{ND_A}$ (Thm.~\ref{thm:8}), \\ $\mathcal{ND_A}$ (Thm.~\ref{thm:3_1}) \end{tabular} & $\mathcal{AN_A}$ (Thm.~\ref{thm:3}) \\ \hline
		Strongly $2^{\mathfrak c}$-algebrable & $\mathcal{AN_A}$ (Thm.~\ref{thm:2}) &   \\ \hline
	\end{tabular}
	\caption{
		This table shows the main results of Section~\ref{sec:2}. 
		Each cell in the right column contains a lineability property $\mathcal L$ that will appear in main theorems, the first row shows the cardinality of the directed sets $\mathcal A$ that will be considered, and the remaining cells containing the families of nets indexed by $\mathcal A$ that satisfy $\mathcal L$.
	}
	\label{tab:1}
\end{table}

%In particular, we will prove: (1) there are a directed sets $\mathcal A$ of cardinality $\cov(\N)$ such that $\mathcal{NM}_{\mathcal A}$, $\mathcal{NF}_{\mathcal A}$ and $\mathcal{ND}_{\mathcal A}$ are positively $2^{\mathfrak c}$-coneable, $\mathcal{ANM}_A\setminus \mathcal{ND_A}$ is $2^{\mathfrak c}$-coneable, and $\mathcal{AN}_{\mathcal A}$ is strongly $2^{\mathfrak c}$-algebrable; (2) there is a directed set $\mathcal A$ of cardinality $\add(\N)$ such that $\mathcal{AN}_{\mathcal A}$ is strongly $\mathfrak c$-algebrable; and (3) there are directed sets $\mathcal A$ of cardinality $\cov(\N)$ such that $\mathcal{ND}_{\mathcal A}$, $\mathcal{ANM}_{\mathcal A}$ and $\mathcal{ANM}_A\setminus \mathcal{ND_A}$ are strongly $\mathfrak c$-algebrable.

In Section~3, we will study the coneability and algebrability of the above $\kappa$-sequences of functions (with $\kappa$ a regular cardinal).
For a summary of the main results, see Table~\ref{tab:2}.
Moreover, as an immediate consequence of the main theorems, we will prove that $\text{LIM}_{\add(\N)}(\mathcal L)\supsetneq \mathcal L$, $\text{LIM}_{\non(\N)}(\mathcal L)\supsetneq \mathcal L$, and also that the claims $\text{LIM}_{\add(\N)}(\mathcal L)=\text{LIM}_{\omega_1}(\mathcal L)$ and $\text{LIM}_{\non(\N)}(\mathcal L)=\text{LIM}_{\omega_1}(\mathcal L)$ are independent of ZFC.

\begin{table}[h!]
	\begin{tabular}{|c|c|c|c|}
		\hline
		& ZFC &  ZFC+$\non(\N)=\mathfrak c$ & ZFC+$\add(\N)=\cov(\N)$ \\ \hline
		\begin{tabular}[c]{@{}c@{}} Positively\\ $2^{\mathfrak c}$-coneable \end{tabular} &  & \begin{tabular}[c]{@{}c@{}} $\mathcal{NM}_{\mathfrak c}$, $\mathcal{NF}_{\mathfrak c}$, $\mathcal{ND}_{\mathfrak c}$\\ (\ref{thm:4}) \end{tabular} & \begin{tabular}[c]{@{}c@{}} $\mathcal{NM}_{\add(\N)}$, $\mathcal{NF}_{\add(\N)}$,\\ $\mathcal{ND}_{\add(\N)}$ (\ref{thm:4_1}) \end{tabular} \\ \hline
		$2^{\mathfrak c}$-coneable &  & \begin{tabular}[c]{@{}c@{}} $\mathcal{ANM}_{\mathfrak c} \setminus \mathcal{ND}_{\mathfrak c}$\\ (\ref{thm:8_1}) \end{tabular} & \begin{tabular}[c]{@{}c@{}} $\mathcal{ANM}_{\add(\N)} \setminus \mathcal{ND}_{\add(\N)}$\\ (\ref{thm:8_2}) \end{tabular} \\ \hline
		\begin{tabular}[c]{@{}c@{}} Strongly\\ $\mathfrak c$-algebrable \end{tabular} & \begin{tabular}[c]{@{}c@{}} $\mathcal{AN}_{\add(\N)}$\\ (\ref{thm:6_1}) \end{tabular} & \begin{tabular}[c]{@{}c@{}} $\mathcal{ANM}_{\mathfrak c} \setminus \mathcal{ND}_{\mathfrak c}$\\ (\ref{thm:9_1});\\ $\mathcal{ND}_{\mathfrak c}$\\ (\ref{thm:6}) \end{tabular} & \begin{tabular}[c]{@{}c@{}} $\mathcal{ANM}_{\add(\N)} \setminus \mathcal{ND}_{\add(\N)}$\\ (\ref{thm:9_2});\\ $\mathcal{ND}_{\add(\N)}$\\ (\ref{thm:6_2}) \end{tabular} \\ \hline
		$2^{\non(\N)}$-algebrable & \begin{tabular}[c]{@{}c@{}} $\mathcal{AN}_{\non(\N)}$\\ (\ref{thm:6_1_3}) \end{tabular} &  & \\ \hline
		\begin{tabular}[c]{@{}c@{}} Strongly\\ $2^{\mathfrak c}$-algebrable \end{tabular} &  & \begin{tabular}[c]{@{}c@{}} $\mathcal{AN}_{\mathfrak c}$\\ (\ref{thm:6_1_1}) \end{tabular} & \begin{tabular}[c]{@{}c@{}} $\mathcal{AN}_{\add(\N)}$\\ (\ref{thm:6_1_2}) \end{tabular} \\ \hline
	\end{tabular}
	\caption{
		This table shows the main results of Section~\ref{sec:2}. 
		Each cell in the right column contains a lineability property $\mathcal L$ that will appear in main theorems, the first row shows the axiomatic $\mathbb V$ that will be considered, and the remaining cells containing the families of uncountable sequences that satisfy $\mathcal L$ in $\mathbb V$.
	}
	\label{tab:2}
\end{table}

Section~\ref{sec:4} is dedicated to final remarks, open problems and possible new lines of research.

%\rotatebox{90}{
%	\begin{tabular}{|c|c|c|c|c|}
%		\hline
%		{} & Positively $2^{\mathfrak c}$-coneable & $2^{\mathfrak c}$-coneable & Strongly ${\mathfrak c}$-algebrable & Strongly $2^{\mathfrak c}$-algebrable \\ 
%		\hline
%		$\card(\A)=\cov(\N)$ & $\mathcal{NM_A}$, $\mathcal{NF_A}$, $\mathcal{ND_A}$ (Thm. ) & $\mathcal{ANM_A}\setminus \mathcal{ND_A}$ (Thm. ) & $\mathcal{ND_A}$ (Thm. ), $\mathcal{ANM_A}\setminus \mathcal{ND_A}$ (Thm. )& $\mathcal{AN_A}$ (Thm. ) \\ 
%		\hline
%	\end{tabular}
%}

\section{On nets of functions}\label{sec:2}

In this paper we will frequently use the so-called method of \textit{independent Bernstein sets} which has been proven to be useful in lineability theory (see  \cites{BaBiGl,BaGlPeSe,Na2}).
To do so, let us define what is known as an independent family of subsets of a given set $X$, and the Fichtenholz-Kantorovich-Hausdorff theorem \cites{FiKa,Ha}.

\begin{definition}\label{def:1}
	Let $X$ be a nonempty set. We say that a family $\mathcal Y$ of subsets of $X$ is independent if for any finite sequences $Y_1,\ldots,Y_n\in \mathcal Y$ of distinct sets
	and $\varepsilon_1,\ldots, \varepsilon_n \in \{0,1\}$
	we have $Y_1^{\varepsilon_1}\cap \cdots \cap Y_n^{\varepsilon_n}\neq \emptyset$, where $Y^1:=Y$ and $Y^0:=X\setminus Y$. 
\end{definition}

\begin{theorem}[Fichtenholz-Kantorovich-Hausdorff theorem]\label{thm:FKH}
	If $X$ is a set of infinite cardinality $\kappa$, then there exists a family of independent subsets $\mathcal Y$ of $X$ of cardinality $2^\kappa$ such that each $Y\in \mathcal Y$ has cardinality $\kappa$.
\end{theorem}

Let $X$ be a Polish space, that is, a separable completely metrizable topological space.
It is known that there are $\mathfrak c$-many disjoint Bernstein subsets of $X$\footnote{If $X$ is an arbitrary Polish space, $\text{Perf}$ is the family of all perfect subsets of $X$, and 
	$\{\la P_\xi,r_\xi \ra \colon \xi<\co\}$ an enumeration of $\text{Perf}\times\R$, choose inductively $\{x_\xi\colon \xi<\co\}\subset X$ so that $x_\xi\in P_\xi\setminus\{x_\eta\colon \eta<\xi\}$ (recall that the cardinality of a perfect subset of a Polish space is $\mathfrak c$, see \cite[corollary~6.3]{Ke}).
	Then the sets $B_r=\{x_\xi\colon r_\xi=r\}$ are disjoint and Bernstein sets.} (see \cite[pp.~105, exercise~1]{Ci}).
Recall that a Bernstein set $B$ in a Polish space $X$ satisfies that $B$ and $X\setminus B$ intersect every perfect subset of $X$.

Given a Polish space $X$, fix $\{B_\xi^X \colon \xi<\mathfrak c \}$ a family of $\mathfrak c$-many pairwise disjoint Bernstein subsets of $X$.
By the Fichtenholz-Kantorovich-Hausdorff theorem, let $\{C_\alpha \colon \alpha< 2^{\mathfrak c} \}$ be an independent family of subsets of $\mathfrak c$ of cardinality $2^{\mathfrak c}$.
For every $\alpha<2^{\mathfrak c}$, let us define
	\begin{equation}\label{equ:1}
		V_\alpha^X:=\bigcup_{\xi \in C_\alpha} B_\xi^X.
	\end{equation}
Observe that $V_\alpha^X$ is also a Bernstein set in $X$.
In the main results of the paper $X$ will be either the real line endowed with the Euclidean metric or any measure $0$ Cantor subset of $\R$
%Cantor set of measure $0$ of $\R$ 
with the inherited topology of $\R$.

Throughout this section we will use the following binary relation.
Let $X$ be a family of subsets of $\R$ such that $X$ has infinite cardinality, and consider $\mathcal A$ the set of all finite subsets of $X$.
We will endow $\mathcal A$ with the following binary relation$\preccurlyeq$:
	\begin{itemize}
		\item[$\circ$]\label{circ} given $A,B\in \mathcal A$ we denote $A\preccurlyeq B$ provided that $\bigcup A\subseteq \bigcup B$.
	\end{itemize}
It is obvious that $\preccurlyeq$ is a preorder on $\mathcal A$ and any pair of elements of $\mathcal A$ has an upper bound in $\mathcal A$.
Indeed, if $A,B\in \mathcal A$, then $A\cup B\in \mathcal A$ by definition, and $\bigcup A, \bigcup B\subseteq \bigcup (A\cup B)$, i.e., $A\preccurlyeq A\cup B$ and $B\preccurlyeq A\cup B$.
Moreover, notice that $\mathcal A$ is not linearly ordered and, since $\mathcal A$ is infinite, it is known that  
$\card(\mathcal A)=\card(X)$.

We are now ready to study the lineability of the families $\mathcal{NM_A}$, $\mathcal{NF_A}$, $\mathcal{AN_A}$, $\mathcal{ND_A}$ and $\mathcal{ANM_A}$.
We begin by studying the existence of positive cones of large dimension inside the families of $\mathcal{NM_A}$, $\mathcal{NF_A}$ and $\mathcal{ND_A}$, where $\mathcal A$ is a directed set of cardinality $\cov(\N)$ but not a cardinal number.
To do so, we present the following general Lemma~\ref{lem:7} for arbitrary directed sets which is also used in Section~\ref{sec:3}.

\begin{lemma}\label{lem:7}
	Let $\mathfrak C\subset \R$ be a Cantor set of measure $0$, $\{V_{\alpha}^{\mathfrak C} \colon \alpha<2^{\mathfrak c} \}$ the sets defined in \eqref{equ:1} at the beginning of Section~\ref{sec:2}, $\mathcal A$ a nonempty directed set and $\left(\sigma_A\right)_{A\in \mathcal A}$ a net of subsets of $\R$.
	If $\sigma_A\cap \mathfrak C=\emptyset$ for every $A\in \mathcal A$, then the nets in 
		$$
		\mathcal B:=\left\{ \left(\chi_{V_\alpha^{\mathfrak C} \cup \sigma_A} \right)_{A\in \mathcal A} \colon \alpha<2^{\mathfrak c} \right\}
		$$ 
	are linearly independent.
	
	Furthermore, if $\sigma_A \in \mathcal N$ is contained in $[0,1]$ for every $A\in \mathcal A$, $\sigma_A\subseteq \sigma_B$ when $A,B\in \mathcal A$ are such that $A\leq B$, $(\chi_{\sigma_A})_{A\in \mathcal A}$ converges pointwise everywhere to $\chi_I$ where $I\subset [0,1]$ is some nondegenerate interval disjoint from $\mathfrak C$, and $\mathfrak C\subset [0,1]$, then the positive cone generated by $\mathcal B$ is contained in $\mathcal{NM_A}$, $\mathcal{NF_A}$ and $\mathcal{ND_A}$.
\end{lemma}

\begin{proof}
	Let $\alpha_1<\cdots<\alpha_n<2^{\mathfrak c}$ be distinct, $\lambda_1,\ldots,\lambda_n\in \mathbb R$ and define the net of functions
		$$
		(F_A)_{A\in \mathcal A}:=\sum_{i=1}^n \lambda_i (\chi_{V_{\alpha_i}^{\mathfrak C} \cup \sigma_A})_{A\in \mathcal A}.
		$$
	
	For the first part, assume that $(F_A)_{A\in \mathcal A}\equiv 0$.
	Then, for every $A\in \mathcal A$, we have that $\sum_{i=1}^n \lambda_i \chi_{V_{\alpha_i}^{\mathfrak C} \cup \sigma_A}=0$.
	By definition of an independent family of subsets, there exists $\xi_0 \in C_{\alpha_1} \setminus \bigcup_{i=2}^n C_{\alpha_i}$.
	Hence, by taking any $x\in B_{\xi_0}^{\mathfrak C}$ we have that $x\in V_{\alpha_1}^{\mathfrak C}\setminus \bigcup_{i=2}^n V_{\alpha_i}^{\mathfrak C}$ since the sets in $\{B_\xi^{\mathfrak C}\colon \xi<\mathfrak c \}$ are pairwise disjoint subsets of $\mathfrak C$.
	Therefore, as $\sigma_A\cap \mathfrak C=\emptyset$ for each $A\in \mathcal A$, $0=\sum_{i=1}^n \lambda_i \chi_{V_{\alpha_i}^{\mathfrak C} \cup \sigma_A}(x)=\lambda_1$ for every $A\in \mathcal A$.
	By repeating the same argument with $\sum_{i=2}^n \lambda_i \chi_{V_{\alpha_i}^{\mathfrak C} \cup \sigma_A}=0$, we arrive at $\lambda_i=0$ for every $i\in \{1,\ldots,n \}$.
	
	For the second part, assume for the rest of the proof that $\lambda_i>0$ for every $i\in \{1,\ldots,n \}$.
	Since the Lebesgue measure in $\R$ is complete and $\mathfrak C$ is a null set, each $V_\alpha^{\mathfrak C}\subset \mathfrak C$ is a null set. 
	Therefore, each $\chi_{V_\alpha^{\mathfrak C} \cup \sigma_A}$ is measurable since $V_\alpha^{\mathfrak C} \cup \sigma_A\in \mathcal N$.
	Thus, $F_A$ is measurable for every $A\in \mathcal A$.
	Moreover, observe that $0\leq F_A\leq F_B$ when $A\leq B$, and $F_A\leq \sum_{i=1}^n \lambda_i \chi_{V_{\alpha_i}^{\mathfrak C} \cup I}$ for every $A\in \mathcal A$.
	It is obvious that $\sum_{i=1}^n \lambda_i \chi_{V_{\alpha_i}^{\mathfrak C} \cup I}$ is integrable but also
		$$
		\liminf_{A\in \mathcal A} \int_{[0,1]} F_A d\lambda=\lim_{A\in \mathcal A} \int_{[0,1]} F_A d\lambda = 0,
		$$
	and
		$$
		\int_{\left[0,1 \right]} \liminf_{A\in \mathcal A} F_A d\lambda =\int_{\left[0,1 \right]} \lim_{A\in \mathcal A} F_A d\lambda = \int_{\left[0,1 \right]} \sum_{i=1}^m \lambda_i \chi_{V_{\alpha_i}^{\mathfrak C} \cup I} d\lambda=\lambda(I)\sum_{i=1}^m \lambda_i>0.
		$$
	Finally, 
		\begin{align*}
			\lim_{A\in \mathcal A}\int_{[0,1]} \left|F_A-\sum_{i=1}^m \lambda_i \chi_{V_\alpha^{\mathfrak C} \cup I}\right| d\lambda & =\lim_{A\in \mathcal A}\int_{\left[0,1 \right]} \sum_{i=1}^m \lambda_i \chi_{I} d\lambda=\lambda(I)\sum_{i=1}^m \lambda_i>0,
		\end{align*}
	Hence, $(F_A)_{A\in \mathcal A}\in \mathcal{NM}_{\mathcal A},\ \mathcal{NF}_{\mathcal A},\ \mathcal{ND_A}$.
\end{proof}

\begin{theorem}\label{thm:1}
	There exists a directed set $\mathcal A$ that is not a cardinal number with $\card(\mathcal A)=\cov(\N)$ such that $\mathcal{NM}_{\mathcal A}$, $\mathcal{NF}_{\mathcal A}$ and $\mathcal{ND}_{\mathcal A}$ in $\left(\mathbb R^{[0,1]} \right)^{\mathcal A}$ are positively 
	$2^{\mathfrak c}$-coneable.
\end{theorem}

\begin{proof}
	Take $X\subset \mathcal N$ witnessing $\cov(\N)$ such that $\bigcup X=\left[0,\frac{1}{2} \right]$.
	Let $\mathcal A$ be the set of all finite subsets of $X$ endowed with the binary relation $\preccurlyeq$ defined in $\circ$ before Lemma~\ref{lem:7}.
	Also, let $\mathfrak C$ be a Cantor set of measure $0$ contained in $\left(\frac{1}{2},1\right]$.
	
	For every $\alpha<2^{\mathfrak c}$, let us define the net $(f_A^\alpha)_{A\in \mathcal A}$ 
	in the following way: for every $A\in \mathcal A$,
		$$
		f_A^\alpha:=\chi_{V_\alpha^{\mathfrak C} \cup \left(\bigcup A \right)},
		$$
	where the sets $V_\alpha^{\mathfrak C}$ are defined in \eqref{equ:1} at the beginning of Section~\ref{sec:2}.
		
	Since $\sigma_A:=\bigcup A\in \N$ for every $A\in \mathcal A$, $\sigma_A\subseteq \sigma_B$ provided that $A,B\in \mathcal A$ are such that $A\preccurlyeq B$, $(\chi_{\sigma_A})_{A\in \mathcal A}$ converges pointwise everywhere to $\chi_I$ where $I:=\left[0,\frac{1}{2} \right]$, and $\sigma_A\cap \mathfrak C=\emptyset$ for every $A\in \mathcal A$, we have the desired result by Lemma~\ref{lem:7}.
\end{proof}

The following main result, Theorem~\ref{thm:7}, shows the existence of cones of large dimension inside the family $\mathcal{ANM_A}\setminus \mathcal{ND_A}$, where $\mathcal A$ is a directed set of cardinality $\add(\N)$ but not a cardinal number.
First let us present general Lemma~\ref{lem:8} for arbitrary directed sets which will be used in Theorem~\ref{thm:7} as well as in Section~\ref{sec:3}.

\begin{lemma}\label{lem:8}
	Let $\mathfrak C\subset \R$ be a Cantor set of measure $0$, $\{V_\alpha^{\mathfrak C} \colon \alpha<2^{\mathfrak c} \}$ the sets defined in \eqref{equ:1} at the beginning of Section~\ref{sec:2}, $\mathcal A$ a nonempty directed set and $\left(\sigma_A\right)_{A\in \mathcal A}$ a net of subsets of $\R$.
	If $\sigma_A\cap \mathfrak C=\emptyset$ for every $A\in \mathcal A$, the nets in $\mathcal B:=\left\{ \left(f_A^\alpha \right)_{A\in \mathcal A} \colon \alpha<2^{\mathfrak c} \right\}$, where
		$$
		f_A^\alpha (x):=\begin{cases}
		\frac{1}{x}\cdot \chi_{V_\alpha^{\mathfrak C} \cup \sigma_A}(x) & \text{if } x\neq 0,\\
		0 & \text{if } x=0,
		\end{cases}
		$$ 
	for every $A\in \mathcal A$ and $\alpha<2^{\mathfrak c}$, are linearly independent.
	
	Furthermore, if $\sigma_A \in \mathcal N$ is contained in $[0,1]$ for every $A\in \mathcal A$, $\sigma_A\subseteq \sigma_B$ when $A,B\in \mathcal A$ are such that $A\leq B$, $(\chi_{\sigma_A})_{A\in \mathcal A}$ converges pointwise everywhere to $\chi_I$ where $I\subset (0,1]$ is a nondegenerate interval disjoint from $\mathfrak C$ and left endpoint 
	$0$,
	and $\mathfrak C\subset [0,1]$, then the positive (resp. negative) cone generated by $\mathcal B$ is contained in $\mathcal{ANM_A}\setminus \mathcal{ND_A}$.
\end{lemma}

\begin{proof}
	Let $\alpha_1,\ldots,\alpha_n<2^{\mathfrak c}$ be distinct, $\lambda_1,\ldots,\lambda_n\in \R$ and define the net
		$$
		(F_A)_{A\in \mathcal A}:=\sum_{i=1}^n \lambda_i (f_A^{\alpha_i})_{A\in \mathcal A}.
		$$
	
	For the first part, assume that $(F_A)_{A\in \mathcal A}\equiv 0$.
	By definition of a family of independent subsets, there exists $\xi_0\in C_{\alpha_1}\setminus \bigcup_{i=2}^n C_{\alpha_i}$, i.e., $B_{\xi_0}^{\mathfrak C}\subset V_{\alpha_1}^{\mathfrak C}\setminus \bigcup_{i=2}^n V_{\alpha_i}^{\mathfrak C}$.
	Hence, by taking any $x\in B_{\xi_0}^{\mathfrak C}\setminus \{0\}$ we have that $F_A(x)=\lambda_1\frac{1}{x}=0$ since $\sigma_A\cap \mathfrak C=\emptyset$ for every $A\in \mathcal A$, that is, $\lambda_1=0$.
	By repeating similar arguments we have that $\lambda_i=0$ for every $i\in \{1,\ldots,n \}$.
	
	For the second part, assume that $\lambda_i>0$ (resp. $\lambda_i<0$) for every $i\in \{1,\ldots,n \}$.
	Since $V_\alpha^{\mathfrak C}\cup \sigma_A\in \N$ for every $A\in \mathcal A$ and $\alpha<2^{\mathfrak c}$, we have that each $f_A^\alpha$ is measurable.
	(Recall that $\mathfrak C\in \N$ and the Lebesgue measure is complete.)
	Therefore, $F_A$ is measurable for every $A\in \mathcal A$.
	Moreover, by Remark~\ref{rem:3}, note that $(F_A)_{A\in \mathcal A}$ converges pointwise everywhere to
		$$
		F(x):=\begin{cases}
		\sum_{i=1}^n \lambda_i \frac{1}{x}\cdot \chi_{V_\alpha^{\mathfrak C} \cup I}(x) & \text{if } x\neq 0,\\
		0 & \text{if } x=0.
		\end{cases}
		$$
	It is obvious that $F$ is a measurable function and, for any $\varepsilon\in \left(0,\sum_{i=1}^m \lambda_i\right)$ (resp. $\varepsilon\in \left(0,-\sum_{i=1}^m \lambda_i\right)$),
		\begin{align*}
			& \lim_{A\in \mathcal A} \lambda\left(\left\{x\in [0,1]\colon \left|F_A(x)-F(x)\right| \geq \varepsilon  \right\} \right)\\
			& =\lim_{A\in \mathcal A} \lambda\left(\left\{x\in I \colon \frac{1}{x} \left|\sum_{i=1}^m \lambda_i \chi_{I}(x)\right| \geq \varepsilon  \right\} \right)=\lambda(I)>0.
		\end{align*}
	However, %\ck{The equality ''$=\lambda(I)$'' is false if 0 is the left endpoint of $I$.\\}
	$F$ is not integrable.
	Indeed, since $I$ is a nondegenerate interval with left endpoint $0$, we have that
		$$
		\int_{[0,1]} |F| d\lambda = \left|\sum_{i=1}^n \lambda_i \right| \int_{I} \frac{1}{x} d\lambda = \infty.
		$$ 
	Thus, $(F_A)_{A\in \mathcal A}\in \mathcal{ANM_A}\setminus \mathcal{ND_A}$.
\end{proof}

\begin{theorem}\label{thm:7}
	There exists a directed set $\mathcal A$ that is not a cardinal number with $\card(\mathcal A)=\cov(\N)$ such that $\mathcal{ANM_A} \setminus \mathcal{ND}_{\mathcal A}$ in $\left(\mathbb R^{[0,1]} \right)^{\mathcal A}$ is $2^{\mathfrak c}$-coneable.
\end{theorem}

\begin{proof}
	Let $X$ be a family of cardinality $\cov(\N)$ of null sets contained in $\left(0,\frac{1}{2} \right]$ whose union is $\left(0,\frac{1}{2} \right]$.
	Consider $\mathcal A$ the set of all finite subsets of $X$ endowed with the binary relation defined in $\circ$ before Lemma~\ref{lem:7}, and $\mathfrak C$ a Cantor set of measure $0$ contained in $\left(\frac{1}{2},1 \right]$.
	Clearly $\card(\mathcal A)=\cov(\N)$ and $\bigcup A\in \N$ for each $A\in \mathcal A$.
	
	For every $\alpha<2^{\mathfrak c}$, let us define the net of functions $(f_A^\alpha)_{A\in \mathcal A}$ as follows: for each $x\in [0,1]$ and $A\in \mathcal A$,
		$$
		f_A^\alpha (x):=\begin{cases}
		\frac{1}{x}\cdot \chi_{V_\alpha^{\mathfrak C} \cup \left(\bigcup A\right)}(x) & \text{if } x\neq 0,\\
		0 & \text{if } x=0,
		\end{cases}
		$$ 
	where the sets $V_\alpha^{\mathfrak C}$ are defined in \eqref{equ:1} at the beginning of Section~\ref{sec:2}.
	
	By taking $\sigma_A:=\bigcup A\in \N$ which satisfies that $\sigma_A\subset \left(0,\frac{1}{2} \right]$ for every $A\in \mathcal A$, we have that $(\chi_A)_{A\in \mathcal A}$ converges pointwise everywhere to $\chi_I$ with $I:=\left(0,\frac{1}{2} \right]$, $\sigma_A\subseteq \sigma_B$ for every $A,B\in \mathcal A$ with $A\preccurlyeq B$, and $\sigma_A\cap \mathfrak C=\emptyset$ for every $A\in \mathcal A$ where $\mathfrak C\subset \left(\frac{1}{2},1 \right]$.
	Hence, by Lemma~\ref{lem:8}, the result follows.
\end{proof}

Let us now study the algebrability of the family $\mathcal{ANM_A}\setminus \mathcal{ND_A}$, where $\mathcal A$ is a specific directed set of cardinality $\cov(\N)$ but not a cardinal number.
In comparison with Theorem~\ref{thm:7}, we can improve the algebraic structure of $\mathcal{ANM_A}\setminus \mathcal{ND_A}$ from a cone to an algebra as it is shown in Theorem~\ref{thm:8}.
However, the cardinality of the generators of the algebraic structure decreases from $2^{\mathfrak c}$ to $\mathfrak c$.
In order to prove Theorem~\ref{thm:8}, we present general Lemma~\ref{lem:9} for arbitrary directed sets which will be used in Section~\ref{sec:3} too.

\begin{lemma}\label{lem:9}
	Let $I$ be a nondegenerate interval, $\mathcal A$ a nonempty directed set, $\left(\sigma_A\right)_{A\in \mathcal A}$ a net of subsets of $\R$ and $\mathcal H$ a Hamel basis of $\R$ over $\mathbb Q$.
	If $\sigma_A\cap I=\emptyset$ for every $A\in \mathcal A$, the nets in $\mathcal B:=\left\{ \left(f_A^h \right)_{A\in \mathcal A} \colon h\in \mathcal H \right\}$, where
		$$
		f_A^h (x):=\begin{cases}
		\frac{e^{hx}}{x} \cdot \chi_{I \cup \sigma_A}(x) & \text{if } x\neq 0,\\
		0 & \text{if } x=0,
		\end{cases}
		$$ 
	for every $A\in \mathcal A$ and $h\in \mathcal H$, are algebraically independent.
	
	Furthermore, if $I=(0,a]$ for some $a<1$, $\sigma_A \subset [0,1]$ is measurable for every $A\in \mathcal A$, $\sigma_A\subseteq \sigma_B$ when $A,B\in \mathcal A$ are such that $A\leq B$, $(\chi_{\sigma_A})_{A\in \mathcal A}$ converges pointwise everywhere to $\chi_J$ where $J\subset [a,1]$ is a nondegenerate compact interval, then the algebra generated by the functions in $\mathcal B$ restricted to $[0,1]$	is
	contained in $(\mathcal{ANM_A}\setminus \mathcal{ND_A})\cup \{0\}$.
\end{lemma}

\begin{proof}
	Let $h_1,\ldots,h_n\in \mathcal H$ be distinct, where $n\in \mathbb N$, and consider the net 
		$$
		(F_A)_{A\in \mathcal A}:= \sum_{i=1}^n \lambda_i \prod_{j=1}^m (f_A^{h_j})_{A\in \mathcal A}^{k_{j,i}},
		$$
	with $\lambda_i\in \mathbb R$ for every $i\in \{1,\ldots,n \}$, $k_{j,i}\in \mathbb N_0$, $\sum_{j=1}^m k_{j,i}\geq 1$ and the $m$-tuples $(k_{1,i},\ldots,k_{m,i})$ are distinct for every $i\in \{1,\ldots,n \}$.
	Then, for any $A\in \mathcal A$ and $x\in \R\setminus \{0\}$, we have that
		$$
		F_A(x)=\sum_{i=1}^n \lambda_i \frac{e^{\sum_{j=1}^m k_{j,i} h_{j} x}}{x^{\sum_{j=1}^m k_{j,i}} } \cdot \chi_{I \cup \sigma_A}(x) = \sum_{i=1}^n \lambda_i \frac{e^{\beta_i x}}{x^{r_i}} \cdot \chi_{I \cup \sigma_A}(x)
		$$
	where $\beta_i:=\sum_{j=1}^m k_{j,i} h_{j}\in \R\setminus \{0\}$ and $r_i:=\sum_{j=1}^m k_{j,i}\in \mathbb N$ for each $i\in \{1,\ldots,n \}$.
	Notice that by construction $\beta_r\neq \beta_s$ for any $r,s\in \{1,\ldots,n \}$ distinct since $\mathcal H$ is a Hamel basis.
	Also observe that $F_A(0)=0$ for every $A\in \mathcal A$.
	
	For the first part, take $r:=\max\{r_i \colon i\in \{1,\ldots,n \} \}\geq 1$.
	Hence, for every $x\in I\setminus \{0\}$, we have
		$$
		F_A(x)=\frac{1}{x^{r+1}} \sum_{i=1}^n \lambda_i x^{s_i} e^{\beta_i x},
		$$
	with $s_i:=r+1-r_i\in \mathbb N$ for every $i\in \{1,\ldots,n \}$.
	Applying the fact that the set of functions in $\{x^m e^{\alpha x} \colon m\in \mathbb N,\ \alpha \in \R \}$ are linearly independent on any nondegenerate interval (see, for instance, \cite[remark~2.4]{FeMuRoSe} or \cite[corollary~2.5]{FeMuRoSe}), the proof of the first part immediately follows.
	
	For the second part, assume that $\lambda_i\neq 0$ for every $i\in \{1,\ldots,n \}$.
	It is obvious that $F_A$ is measurable for every $A\in \mathcal A$ since each $\sigma_A$ is measurable.
	Moreover, since $(\chi_{\sigma_A})_{A\in \mathcal A}$ converges pointwise everywhere to $\chi_J$, we have that $(F_A\restriction [0,1])_{A\in \mathcal A}$ converges pointwise everywhere to 
		$$
		F(x):=\begin{cases}
		\sum_{i=1}^n \lambda_i \frac{e^{\beta_i x}}{x^{r_i}} \cdot \chi_{I \cup J}(x) & \text{if } x\in (0,1],\\
		0 & \text{if } x=0.
		\end{cases}
		$$
	Now, as $0\notin J$ is a compact interval and $g(x):=\left|\sum_{i=1}^n \lambda_i \frac{e^{\beta_i x}}{x^{r_i}}\right|$ is continuous function on any compact interval that does not contain $0$, there is an $\overline{x}\in J$ such that 
		$$
		g(\overline{x})=\min\left\{ g(x) \colon x\in J \right\}. 
		$$
	Thus, for every $\varepsilon \in (0,g(\overline{x}))$, we have by construction
		\begin{align*}
			& \lim_{A\in \mathcal A} \lambda\left(\left\{x\in [0,1]\colon \left|F_A(x)-F(x)\right| \geq \varepsilon  \right\} \right)\\
			& =\lim_{A\in \mathcal A} \lambda\left(\left\{x\in J \colon \left|	\sum_{i=1}^n \lambda_i \frac{e^{\beta_i x}}{x^{r_i}} \cdot \chi_{J}(x)\right| \geq \varepsilon  \right\} \right)=\lambda(J)>0.
		\end{align*}
	(Recall that $J$ is nondegenerate.)
	Finally, observe that
		$$
		\int_{[0,1]} |F| d\lambda \geq \int_0^a \frac{1}{x^{r+1}} \left|\sum_{i=1}^n \lambda_i x^{s_i} e^{\beta_i x}\right| d\lambda = \infty,
		$$
	where we have used the fact that $r+1>s_i\geq 1$ for every $i\in \{1,\ldots,n\}$.
	We have finished the proof.
\end{proof}

\begin{theorem}\label{thm:8}
	There exists a directed set $\mathcal A$ that is not a cardinal number with $\card(\mathcal A)=\cov(\N)$ such that $\mathcal{ANM_A} \setminus \mathcal{ND}_{\mathcal A}$ in $\left(\mathbb R^{[0,1]} \right)^{\mathcal A}$ is strongly $\mathfrak c$-algebrable.
\end{theorem}

\begin{proof}
	Let $\mathcal H$ be a Hamel basis of $\R$ over $\mathbb Q$ and, for every $h\in \mathcal H$, let 
		$$
		f_h(x):=\frac{e^{h x}}{x},
		$$
	for any $x\in [0,1]$.
	Recall that $\text{card}(\mathcal H)=\mathfrak c$.
	
	Take $X\subset \mathcal N$ witnessing $\cov(\N)$ such that $\bigcup X=\left[\frac{2}{3},1 \right]$.
	Let $\mathcal A$ be the set of all finite subsets of $X$ endowed with the binary relation $\preccurlyeq$ defined in $\circ$ before Lemma~\ref{lem:7}.
	Recall that $\card(\mathcal A)=\card(X)=\cov(\N)$.
	
	For every $h\in \mathcal H$, let $(f_A^h)_{A\in \mathcal A}\in \left(\mathbb R^{[0,1]} \right)^{\mathcal A}$ be defined as follows: for each $A\in \mathcal A$,
		$$
		f_A^h(x):=\begin{cases}
		f_h(x) \cdot \chi_{\left(0,\frac{1}{2} \right] \cup \left(\bigcup A \right)}(x) & \text{if } x\neq 0, \\
		0 & \text{if } x=0.
		\end{cases}
		$$
	
	Since $\sigma_A:= \bigcup A \in \N$ for every $A\in \mathcal A$, $\sigma_A \subset J$ with $J:=\left[\frac{2}{3},1 \right]$, for every $A\in \mathcal A$, $\sigma_A\cap I=\emptyset$ for each $A\in \mathcal A$ with $I=\left(0,\frac{1}{2} \right]$, and $(\chi_{\sigma_A})_{A\in \mathcal A}$ converges pointwise everywhere to $\chi_J$, we have the desired result by Lemma~\ref{lem:9}.
\end{proof}

We will now study the algebrability of $\mathcal{ND_A}$ and $\mathcal{AN_A}$.
Similarly to Theorems~\ref{thm:7} and~\ref{thm:8}, we can improve the algebraic structure of the family $\mathcal{ND_A}$ from a positive cone (see Theorem~\ref{thm:1}) to an algebra.
This is shown in the following Theorem~\ref{thm:3_1}, in which we will prove that $\mathcal{ND}_{\mathcal A}$ is strongly $\mathfrak c$-algebrable for a directed set $\mathcal{A}$ of cardinality $\cov(\N)$ that is not a cardinal number.
Still, as in the case of $\mathcal{ANM_A}\setminus \mathcal{ND_A}$, 
the cardinality of the number of generators of the algebraic structure decreases from $2^{\mathfrak c}$ to $\mathfrak c$.
In the case of $\mathcal{AN_A}$ we will prove that $\mathcal{AN_A}$ is strongly $\mathfrak c$-algebrable for a directed set $\mathcal A$ of cardinality $\add(\N)$ that is not a cardinal number.
To do so, we begin by proving the following lemmas.

\begin{lemma}\label{lem:2}
	If $X$ is a family of pairwise disjoint null sets of $\R$ witnessing $\add(\N)$, then there is a  nonmeasurable $V\subset \bigcup X$ that has inner measure $0$ and same outer measure as $\bigcup X$.
\end{lemma}

\begin{proof}
	Let F be an $F_\sigma$-subset of $\bigcup X$ with the same inner measure as $\bigcup X$.
	If $F$ has positive measure, then take $B$ a Bernstein subset of $F$. 
	Otherwise, take $B = \emptyset$. 
	Then $V := B \cup \left(\bigcup X \setminus F \right)$ is as needed, as it has inner measure 0 and the same outer measure as $\bigcup X$.
	In particular, V is non-measurable.
\end{proof}

\begin{lemma}\label{lem:3}
	Let $n\in \mathbb N$, $\lambda_1,\ldots,\lambda_n\in \mathbb R\setminus \{0\}$, $\beta_1,\ldots,\beta_n\in \R\setminus \{0\}$ be pairwise distinct and $I$ a nondegenerate interval.
	If $\sum_{i=1}^n \lambda_i e^{\beta_i x}=0$ a.e. on $I$, then $\lambda_i=0$ for every $i\in \{1,\ldots,n \}$.
\end{lemma}

\begin{proof}
	Define the function $g\colon \mathbb C\to \mathbb C$ by $g(z)=\sum_{i=1}^n \lambda_i e^{\beta_i z}$.
	It is easy to see that $g$ is a holomorphic function and the set of zeros of $g$ has an accumulation point ($\sum_{i=1}^n \lambda_i e^{\beta_i x}$ is equal to $0$ a.e. on $I$). 
	Thus, by the Identity Theorem, we have that $g\equiv 0$.
	Since the exponential functions with nonzero distinct exponents are linearly independent on the complex plane, we have that $\lambda_i=0$ for every $i\in \{1,\ldots,n \}$, a contradiction.
\end{proof}

\begin{lemma}\label{lem:4}
	Let $n\in \mathbb N$, $\lambda_1,\ldots,\lambda_n\in \mathbb R \setminus \{0\}$, $\beta_1,\ldots,\beta_n\in \R\setminus \{0\}$ be pairwise distinct and $I$ a nondegenerate interval different from $\R$. If $V\subset \R\setminus I$ is a nonmeasurable set having inner measure $0$, then $\sum_{i=1}^n \lambda_i e^{\beta_i x} \cdot \chi_{V\cup I}(x)$ is nonmeasurable.
\end{lemma}

\begin{proof}
	Let us show first that $\sum_{i=1}^n \lambda_i e^{\beta_i x} \cdot \chi_V(x)$ is nonmeasurable.
	Assume otherwise.
	Then, since $V$ has inner measure $0$, $\sum_{i=1}^n \lambda_i e^{\beta_i x}$ is equal to $0$ a.e.
	Hence, by Lemma~\ref{lem:3}, $\lambda_i=0$ for every $i\in \{1,\ldots,n \}$, contradicting our hypotheses.
	
	Now $\sum_{i=1}^n \lambda_i e^{\beta_i x} \cdot \chi_{V\cup I}(x)$ is nonmeasurable since, assuming otherwise, we would have that 
	$$
	\sum_{i=1}^n \lambda_i e^{\beta_i x} \cdot \chi_V(x)=\sum_{i=1}^n \lambda_i e^{\beta_i x}  \cdot \chi_{V\cup I}(x)-\sum_{i=1}^n \lambda_i e^{\beta_i x} \cdot \chi_I(x)
	$$
	would be measurable, a contradiction.
\end{proof}

\begin{lemma}\label{lem:5}
	Let $\mathcal H$ be a Hamel basis of $\R$ over $\mathbb Q$, $I,J$ nondegenerate intervals with $J\setminus I$ having nonempty interior and $\mathcal A$ a nonempty directed set.
	Assume that $(\sigma_A)_{A\in \mathcal A}$ is a net of subets of $\R$ such that $\sigma_A\subseteq J\setminus I$ for every $A\in \mathcal A$ and define, for every $h\in \mathcal H$ and $A\in \mathcal A$, the net
	$$
	f_A^h:=f_h \cdot \chi_{\sigma_A\cup I},
	$$
	where $f_h(x)=e^{hx}$ for every $x\in \R$.
	Then the nets of functions in $\{(f_A^h)_{A\in \mathcal A} \colon h\in \mathcal H \}$ are algebraically independent.
	
	Moreover, assume that $\sigma_A\in \N$ for every $A\in \mathcal A$.
	\begin{itemize}
		\item[(i)] If $V\subset J\setminus I$ is a nonmeasurable set with inner measure $0$ such that $(\chi_{\sigma_A})_{A\in \mathcal A}$ converges pointwise everyhwere to $\chi_V$, then the algebra generated by $\{(f_A^h)_{A\in \mathcal A} \colon h\in \mathcal H \}$ is contained in $\mathcal{AN_\mathcal A}\cup \{0\}$.
		
		\item[(ii)] If $J,I\subset [0,1]$ and $(\chi_{\sigma_A})_{A\in \mathcal A}$ converges pointwise everyhwere to $\chi_J$, then the algebra generated by $\{(f_A^h\restriction [0,1])_{A\in \mathcal A} \colon h\in \mathcal H \}$ is contained in $\mathcal{ND_\mathcal A}\cup \{0\}$.
	\end{itemize}
\end{lemma}

\begin{proof}
	Take finitely many distinct $h_1,\ldots,h_m\in \mathcal H$ and consider the net 
	$$
	(F_A)_{A\in \mathcal A}:= \sum_{i=1}^n \lambda_i \prod_{j=1}^m (f_A^{h_j})_{A\in \mathcal A}^{k_{j,i}},
	$$
	where $n\in \mathbb N$, $\lambda_i\in \mathbb R$ for every $i\in \{1,\ldots,n \}$, $k_{j,i}\in \mathbb N_0$, $\sum_{j=1}^m k_{j,i}\geq 1$ and the $m$-tuples $(k_{1,i},\ldots,k_{m,i})$ are distinct for every $i\in \{1,\ldots,n \}$.
	Then, for any $A\in \mathcal A$ and $x\in \R$, we have that
	$$
	F_A(x)=\sum_{i=1}^n \lambda_i e^{\sum_{j=1}^m k_{j,i} h_{j} x} \cdot \chi_{\sigma_A\cup I}(x)=\sum_{i=1}^n \lambda_i e^{\beta_i x} \cdot \chi_{\sigma_A\cup I}(x)
	$$
	where $\beta_i:=\sum_{j=1}^m k_{j,i} h_{j}\in \R\setminus \{0\}$.
	Notice that by construction $\beta_r\neq \beta_s$ for any $r,s\in \{1,\ldots,n \}$ distinct since $\mathcal H$ is a Hamel basis.
	
	For the first part, assume that $(F_A)_{A\in \mathcal A}=0$.
	As the exponential functions with nonzero distinct exponents are linearly independent on any nondegenerate interval, we have that $\lambda_i=0$ for every $i\in \{1,\ldots,n \}$.
	Therefore the nets in $\{(f_A^h)_{A\in \mathcal A} \colon h\in \mathcal H \}$ are algebraically independent.
	
	For the last part, observe that for every $h\in \mathcal H$, the net of functions $(f_A^h)_{A\in \mathcal A}$ converges pointwise everywhere to either $f_h \cdot \chi_{V\cup I}$ for part (i) (which is clearly nonmeasurable since the union of two disjoint sets one being nonmeasurable and the other having positive measure is nonmeasurable) or the integrable function $f_h\cdot \chi_{J\cup I}$ for part (ii).
	Clearly, in any case, each $f_A^h$ is measurable since $A$ is a null set and $f_h$ is measurable.
	
	For part (i) it is clear that each $F_A$ is measurable but also the net $(F_A)_{A\in \mathcal A}$ converges pointwise everywhere to 
	$$
	F(x)=\sum_{i=1}^n \lambda_i e^{\beta_i x} \cdot \chi_{V \cup I}(x).
	$$
	Since $V$ is nonmeasurable and has inner measure $0$ we have, by Lemma~\ref{lem:4}, that $F$ is nonmeasurable.
	
	For part (ii), $0\leq f_A^h \restriction [0,1] \leq f_h \restriction [0,1]$ a.e. for every $A\in \mathcal A$.
	Furthermore, notice that by construction, the net of functions $(f_A^h \restriction [0,1])_{A\in \mathcal A}$ converges pointwise everywhere to the integrable function $f_h \restriction [0,1]$ but
	\begin{align*}
	\lim_{A\in \mathcal A}\int_{[0,1]} |f_A^h-f_h | d\lambda & =\lim_{A\in \mathcal A}\int_{J\setminus I} |f_h| d\lambda \neq 0.
	\end{align*}
	Thus, $(f_A^h \restriction [0,1])_{A\in \mathcal A}\in \mathcal{ND}_{\mathcal A}$.
	Now, clearly every $F_A \restriction [0,1]$ is measurable and also bounded a.e. by $\left|\sum_{i=1}^n \lambda_i e^{\beta_i x}\right|$ for every $x\in [0,1]$ (note that we have $\int_{[0,1]} \left|\sum_{i=1}^n \lambda_i e^{\beta_i x}\right| d\lambda<\infty$).
	Moreover, the net $(F_A \restriction [0,1])_{A\in \mathcal A}$ converges pointwise everywhere to the integrable function $\sum_{i=1}^n \lambda_i e^{\beta_i x}\cdot \chi_{J\cup I}$ which satisfies that
	\begin{align*}
	\lim_{A\in \mathcal A}\int_{[0,1]} \left|F_{A}-\sum_{i=1}^n \lambda_i e^{\beta_i x}\right| d\lambda & =\lim_{A\in \mathcal A} \int_{J\setminus I} \left|\sum_{i=1}^n \lambda_i e^{\beta_i x}\right| d\lambda\\
	& = \int_{J\setminus I} \left|\sum_{i=1}^n \lambda_i e^{\beta_i x}\right| d\lambda\neq 0.
	\end{align*}
	Indeed, if $\int_{J\setminus I} \left|\sum_{i=1}^n \lambda_i e^{\beta_i x}\right| d\lambda$ were equal to $0$, then $\sum_{i=1}^n \lambda_i e^{\beta_i x}$ would be equal to $0$ a.e. on $J\setminus I$.
	Hence, by Lemma~\ref{lem:3}, we would have that $\lambda_i=0$ for every $i\in \{1,\ldots,n \}$, 
	a contradiction. 
\end{proof}

\begin{theorem}\label{thm:3_1} 
	There exists a directed set $\mathcal A$ that is not a cardinal number with $\card(\mathcal A)=\cov(\N)$ such that $\mathcal{ND}_{\mathcal A}$ in $\left( \mathbb{R}^{[0,1]}\right)^{\mathcal A}$ is strongly $\mathfrak c$-algebrable.
\end{theorem}

\begin{proof}
	Let $\mathcal H$ be a Hamel basis of $\R$ over $\mathbb Q$ and, for every $h\in \mathcal H$, let 
	$$
	f_h(x):=e^{h x},
	$$
	for any $x\in [0,1]$.
	
	Take $X\subset \mathcal N$ witnessing $\cov(\N)$ such that $\bigcup X=\left[0,\frac{1}{2} \right]$.
	Let $\mathcal A$ be the set of all finite subsets of $X$ endowed with the binary relation $\preccurlyeq$ defined in $\circ$ before Lemma~\ref{lem:7}.
	Recall that $\card(\mathcal A)=\card(X)=\cov(\N)$.
	
	For every $h\in \mathcal H$, 
	let $(f_A^h)_{A\in \mathcal A}\in \left(\mathbb R^{[0,1]} \right)^{\mathcal A}$ be defined as follows: for each $A\in \mathcal A$,
	$$
	f_A^h:=f_h \cdot \chi_{\left(\frac{1}{2},1 \right] \cup \left(\bigcup A \right)}.
	$$
	
	Since $\sigma_A:= \bigcup A \in \N$ for every $A\in \mathcal A$, $\sigma_A \subset J\setminus I$ with $J:=\left[0,\frac{1}{2} \right]$ and $I:=\left(\frac{1}{2},1 \right]$, for every $A\in \mathcal A$, and $(\chi_{\sigma_A})_{A\in \mathcal A}$ converges pointwise everywhere to $\chi_J$, we have the desired result by the first part and the additional part (ii) of Lemma~\ref{lem:5}. 
\end{proof}

\begin{theorem}\label{thm:3} 
	There exists a directed set $\mathcal A$ that is not a cardinal number with $\card(\mathcal A)=\add(\N)$ such that $\mathcal{AN}_{\mathcal A}$ in $\left(\mathbb R^\R \right)^{\mathcal A}$ is strongly $\mathfrak c$-algebrable.
\end{theorem}

\begin{proof}
	Let $\mathcal H$ be a Hamel basis of $\R$ over $\mathbb Q$.
	For every $h\in \mathcal H$, let 
	$$
	f_h(x):=e^{h x},
	$$
	for any $x\in \R$.
	
	Take $X\subset \N$ witnessing $\add(\N)$ such that each set that belongs to $X$ is contained in 
	$[0,1]$ and the sets of $X$ are pairwise disjoint.
	By Lemma~\ref{lem:2}, there is a nonmeasurable $V\subset \bigcup X$ that has inner measure $0$ and same outer measure as $\bigcup X$. 
	Now let $X^\prime$ be the family $\{Y\cap V \colon Y\in X \}$.
	Cleary $V=\bigcup X^\prime$.
	
	Let $\mathcal A$ be the set of all finite subsets of $X^\prime$ endowed with the binary relation $\preccurlyeq$ defined in $\circ$ before Lemma~\ref{lem:7}.
	Notice that $\card(\mathcal A)=\card(X^\prime)= \card(X)=\add(\N)$.
	Indeed, observe that $X^\prime$ is infinite since otherwise $V$ would be a null set, so $\card(\mathcal A)=\card(X^\prime)$.
	On the other hand, it is clear that $\card(X^\prime)\leq \add(\N)$.
	If we assume that $\card(X^\prime)< \add(\N)$, then $V\in \N$ by definition of $\add(\N)$, which is absurd.
	
	For every $h\in \mathcal H$, let $(f_A^h)_{A\in \mathcal A}\in \left(\mathbb R^\R \right)^{\mathcal A}$
	be defined as follows: for each $A\in \mathcal A$,
	$$
	f_A^h:=f_h \cdot \chi_{(1,2) \cup \left(\bigcup A \right)}.
	$$
	Since $\sigma_A:=\bigcup A \in \N$ for every $A\in \mathcal A$, $\sigma_A\subset J\setminus I$ with $J:=[0,1]$ and $I:=(1,2)$, for every $A\in \mathcal A$, and $(\chi_{\sigma_A})_{A\in \mathcal A}$ converges pointwise everyhwere to $\chi_V$ where $V\subset J\setminus I$ is nonmeasurable and has inner measure $0$, we have the desired result by the first part and the additional part (i) of Lemma~\ref{lem:5}.
\end{proof}

Finally, to finish this section, we will prove in Theorem~\ref{thm:2} that by (maybe) increasing the size of the directed set $\mathcal A$ in $\mathcal{AN_A}$ we can improve the algebraic structure of Theorem~\ref{thm:3}.
In particular, we will show that there is a directed set $\mathcal A$ of cardinality $\cov(\N)$ that is not a cardinal number such that $\mathcal{AN_A}$ is strongly $2^{\mathfrak c}$-algebrable.

For Theorems~\ref{thm:2}, \ref{thm:6_1_1} and~\ref{thm:6_1_2}, we will use Theorem~\ref{CiRoSe} and Lemma~\ref{lem:1} below. 
Theorem~\ref{CiRoSe} can be deduced from \cite{CiRoSe} and, for simplicity, we will use the notation from \cite{CiRoSe}.
For any $n\in \mathbb N$, let $\mathcal P^n$ denote the family of polynomials in $n$ variables, with real coefficients and without independent term; and $n^\mathbb N:=\{0,\ldots,n-1 \}^{\mathbb N}$.
Observe that $\mathcal P:=\bigcup_{n\in \mathbb N} \mathcal P^n \times n^{\mathbb N}$ has cardinality $\mathfrak c$ and the Stone-Čech compactification $\beta \mathbb N$ of $\mathbb N$ has cardinality $2^{\mathfrak c}$ (see \cite{En} for more information on the Stone-Čech compactification).

\begin{theorem}[\cite{CiRoSe}]\label{CiRoSe}
	There exists a family of $2^{\mathfrak c}$-many algebraically independent functions $\{g_{\mathcal U}\colon \mathcal U\in \beta \mathbb N \}$ and $\mathfrak c$-many distinct real numbers $\{x_{P,q} \colon \la P,q \ra\in \mathcal P \}$ such that:	
		\begin{itemize}
			\item[$\star$] For distinct $\mathcal U_1,\ldots,\mathcal U_n\in \beta \mathbb N$ and $P\in \mathcal P^n$, there exists $q\in n^{\mathbb N}$ with $P(g_{\mathcal U_1},\ldots,g_{\mathcal U_n})(x_{P,q})\neq 0$.
		\end{itemize}
\end{theorem}

The existence of the family $\{x_{P,q} \colon \la P,q \ra\in \mathcal P \}$ can be found in \cite[lemma 2.5]{CiRoSe}, and the family of algebraic independent functions $\{g_{\mathcal U}\colon \mathcal U\in \beta \mathbb N \}$ can be deduced from the proof of \cite[theorem 3.1]{CiRoSe}.

Since $\card(\mathfrak c\times \mathcal P)=\mathfrak c$, there exists a family $\{B_{\xi,P,q}^\R \colon \la \xi, \la P,q \ra \ra \in \mathfrak c\times \mathcal P \}$ of disjoint Bernstein subsets of $\R$.
For any $\xi<\mathfrak c$, let $B_\xi^\R:=\bigcup_{\la P,q \ra \in \mathcal P} B_{\xi,P,q}^\R$.
Clearly $\{B_\xi^\R \colon \xi< \mathfrak c \}$ is also a family of disjoint Bernstein subsets of $\R$.
Define $\varphi \colon \R\to \R$ so that  $\varphi[B_{\xi,P,q}^\R]=\{x_{P,q} \}$ for  
every $\la \xi,\la P,q \ra \ra \in \mathfrak c\times \mathcal P$.

\begin{lemma}\label{lem:1}
	Let $\psi\colon \beta \mathbb N \to 2^{\mathfrak c}$ be a bijection, $V_{\psi(\mathcal U)}^\R$ the sets defined in \eqref{equ:1} at the beginning of Section~\ref{sec:2}, $\mathcal A$ a directed set and $\left(\sigma_A^{\mathcal U}\right)_{A\in \mathcal A}$ a net of subsets of $\R$ for every $\mathcal U\in \beta \mathbb N$.
	If $\left(\chi_{\sigma_A^{\mathcal U}}\right)_{A\in \mathcal A}$ converges pointwise everywhere to $\chi_{V_{\psi(\mathcal U)}^\R}$ for every $\mathcal U\in \beta \mathbb N$, then the nets in 
		$$
		\mathcal S:=\left\{ \left((g_{\mathcal U} \circ \varphi) \cdot \chi_{\sigma_A^{\mathcal U}}\right)_{A\in \mathcal A} \colon \mathcal U\in \beta\mathbb N \right\}
		$$ 
	are algebraically independent, where $\varphi$ is the function defined in the paragraph before Lemma~\ref{lem:1} and the functions in $\{g_{\mathcal U} \colon \mathcal U\in \beta \mathbb N \}$ satisfy Theorem~\ref{CiRoSe}.
	
	Moreover, if $\sigma_A^\mathcal U\in \mathcal N$ for every $A\in \mathcal A$ and $\mathcal U\in \beta \mathbb N$, then the algebra generated by $\mathcal S$ is contained in $\mathcal{AN_A}\cup \{0\}$.
\end{lemma}

\begin{proof}
	Let $\mathcal U_1,\ldots,\mathcal U_n \in \beta \mathbb N$ be distinct.
	We will begin by proving the first part.
	To do so, we will prove first the following:
	\begin{itemize}
		\item[$\triangle$] If 
			$$
			x_0\in V_m:=\bigcap_{j=1}^m V_{\psi(\mathcal U_{i_j})}^\R \setminus \bigcup_{i\in \{1,\ldots,n \}\setminus \{i_1,\ldots,i_m \}} V_{\psi(\mathcal U_i)}^\R,
			$$
		with $m\in \{1,\ldots,n \}$ and $i_1,\ldots,i_m\in \{1,\ldots,n \}$ distinct, there is an $A_0\in \mathcal A$ such that 
			$$
			x_0\in V_m^A:=\bigcap_{j=1}^m \sigma_A^{\mathcal U_{i_j}} \setminus \bigcup_{i\in \{1,\ldots,n \}\setminus \{i_1,\ldots,i_m \}} \sigma_A^{\mathcal U_i}
			$$ 
		for every $A\in \mathcal A$ with $A\geq A_0$.
		If $x_0\in \R\setminus \bigcup_{i=1}^n V_{\psi(\mathcal U_i)}^\R$, then there is an $A_0\in \mathcal A$ such that $x_0 \notin \sigma_A^{\mathcal U_i}$ for every $A\in \mathcal A$ with  $A\geq A_0$ and every $i\in \{1,\ldots,n \}$.
	\end{itemize}
	If $x_0\in V_m$, then, since $\left(\chi_{\sigma_A^{\mathcal U_i}}\right)_{A\in \mathcal A}$ converges pointwise everywhere to $\chi_{V_{\psi(\mathcal U_i)}^\R}$ for every $i\in \{1,\ldots,n \}$, we have that for every $j\in \{1,\ldots,m \}$ there is an $A_{i_j} \in \mathcal A$ such that $x_0\in \sigma_A^{\mathcal U_{i_j}}$ for every $A\in \mathcal A$ with $A\geq A_{i_j}$, and for every $i\in \{1,\ldots,n \}\setminus \{1,\ldots,m \}$ there is an $A_i \in \mathcal A$ such that $x_0\notin \sigma_A^{\mathcal U_i}$ for every $A\in \mathcal A$ with $A\geq A_i$.
	Now, as $\mathcal A$ is directed set we can take $A_0\in \mathcal A$ such that $A_0\geq A_i$ for every $i\in \{1,\ldots,n \}$.
	Hence, $x_0\in V_m^A$ for every $A\geq A_0$.
	Analogously observe that if $x_0\in \R\setminus \bigcup_{i=1}^n V_{\psi(\mathcal U_i)}^\R$, then there is an $A_0\in \mathcal A$ such that $x_0 \notin \sigma_A^{\mathcal U_i}$ for every $A\geq A_0$ and $i\in \{1,\ldots,n \}$.
	
	Let $P\in \mathcal P^n$.
	By definition of a family of independent subsets, notice that there exists $\xi<\mathfrak c$ such that $B_\xi^\R \subset \bigcap_{i=1}^n V_{\psi(\mathcal U_i)}^\R$.
	Let $q\in n^{\mathbb N}$ witness $\star$ in Theorem~\ref{CiRoSe} and take any arbitrary $x\in B_{\xi,P,q}^\R \subset B_\xi^\R \subset \bigcap_{i=1}^n V_{\psi(\mathcal U_i)}^\R$.
	Recall that $\varphi(x)=x_{P,q}$.
	By $\star$ in Theorem~\ref{CiRoSe} and $\triangle$, there is an $A_0\in \mathcal A$ such that
	\begin{align*}
		P(g_{\mathcal U_1} (\varphi(x)) \cdot \chi_{\sigma_{A_0}^{\mathcal U_1}}(x),\ldots,g_{\mathcal U_n} (\varphi(x)) \cdot \chi_{\sigma_{A_0}^{\mathcal U_n}}(x))= P(g_{\mathcal U_1} (x_{P,q}),\ldots,g_{\mathcal U_n} (x_{P,q}))\neq 0.
	\end{align*}

	For the second part, observe that
		\begin{equation}\label{equ:2}
			\left(P\left(\left(g_{\mathcal U_1} \circ \varphi\right) \cdot \chi_{\sigma_A^{\mathcal U_1}},\ldots,\left(g_{\mathcal U_n} \circ \varphi\right) \cdot \chi_{\sigma_A^{\mathcal U_n}}\right)\right)_{A\in \mathcal A}
		\end{equation}
	is a net of measurable functions since $\sigma_A^{\mathcal U_i}\in \N$ for every $i\in \{1,\ldots,n \}$ and $A\in \mathcal A$.
	Moreover, by $\triangle$, the net \eqref{equ:2} converges pointwise everywhere to 
		$$
		F:=P\left(\left(g_{\mathcal U_1} \circ \varphi\right) \cdot \chi_{V_{\psi(\mathcal U_1)}^\R},\ldots,\left(g_{\mathcal U_n} \circ \varphi\right) \cdot \chi_{V_{\psi(\mathcal U_n)}^\R}\right).
		$$
	It remains to prove that $F$ is nonmeasurable.
	%	\ck{I did not check all details of the proof of Thm 2.7. But it looks OK.\\}
	Note that by construction $F(x)\neq 0$ for any $x\in B_{\xi,P,q}^\R$, that is, $B_{\xi,P,q}^\R \subseteq \R \setminus F^{-1}(0)$.
	Once again by definition of a family of independent subsets, there exists $\xi_0<\mathfrak c$ such that $B_{\xi_0}^\R \subset \R\setminus \bigcup_{i=1}^n V_{\psi(\mathcal U_i)}^\R$. 
	Hence, $B_{\xi_0}^\R\cap B_{\xi,P,q}^\R= \emptyset$ and $B_{\xi_0}^\R \subseteq F^{-1}(0)$.
	As $B_\xi$ and $B_{\xi_0}$ are disjoint and have full outer measure in $\R$, $F^{-1}(0)$ is nonmeasurable, so $F$ is nonmeasurable.
\end{proof}

\begin{theorem}\label{thm:2}
	There exists a directed set $\mathcal A$ that is not a cardinal number with $\card(\mathcal A)=\cov(\N)$ such that $\mathcal{AN}_{\mathcal A}$ in $\left(\mathbb R^\R \right)^{\mathcal A}$ is strongly $2^{\mathfrak c}$-algebrable.
\end{theorem}

\begin{proof}
	Let $\psi \colon \beta \mathbb N\to 2^{\mathfrak c}$ be a bijection.
	Take $X\subset \N$ witnessing $\cov(\N)$ such that $\R=\bigcup X$.
	Let $\mathcal A$ be the family of all finite subsets of $X$ endowed with the binary relation $\preccurlyeq$ defined in $\circ$ before Lemma~\ref{lem:7}.
	Hence $\mathcal A$ is a directed set that is not linearly ordered such that $\card(\mathcal A)=\cov(\N)$.
	For every $\mathcal U\in \beta \mathbb N$, let $(f_A^{\mathcal U})_{A\in \mathcal A}\subset \R^\R$ be 
	the
%	\ck{Something wrong:  $V_{\psi(\mathcal U)}^\R$ is not correctly defined, since 
%	\\} 
	following net: for each $A\in \mathcal A$ and every $x\in \R$,
		$$
		f_A^{\mathcal U}(x)=g_{\mathcal U} (\varphi(x)) \cdot \chi_{V_{\psi(\mathcal U)}^\R\cap \left(\bigcup A\right)}(x),
		$$
	where the sets $V_{\psi(\mathcal U)}^\R$ are defined in \eqref{equ:1} at the beginning of Section~\ref{sec:2}, the functions $g_{\mathcal U}$ satisfy Theorem~\ref{CiRoSe}, and $\varphi$ is defined in the paragraph before Lemma~\ref{lem:1}.
	
	Since $\sigma_A^{\mathcal U}:=V_{\psi(\mathcal U)}^\R\cap \left(\bigcup A\right)\in \mathcal N$ for every $A\in \mathcal A$ and $\mathcal U\in \beta\mathbb N$, and $\left(\chi_{\sigma_A^{\mathcal U}}\right)_{A\in \mathcal A}$ converges pointwise everywhere to $\chi_{V_{\psi(\mathcal U)}^\R}$, we have the desired result by Lemma~\ref{lem:1}.
\end{proof}

\section{On uncountable sequences of functions}\label{sec:3}

Let us begin by showing that only $\kappa$-sequences when $\kappa$ is a regular cardinal are of true interest.

\begin{proposition}\label{prop:1}
	Let $\kappa$ be an infinite cardinal number, $X$ a topological space, $(x_\alpha)_{\alpha<\kappa}$ a $\kappa$-sequence in $X$ and
%	\ck{The proposition is true for ANY topological space $X$.\\} 
%	\ch{\textcolor{blue}{Yes, I know. I only considered metric spaces because it is the setting that I am working on. I changed the definition of convergence of transfinite sequences to include arbitrary topological spaces and Proposition \ref{prop:1}.}\\}
		\begin{align*}
		f \colon \cof(\kappa) &\to \kappa\\
		\xi &\mapsto f(\xi)=:\alpha_\xi
		\end{align*}
	a strictly increasing cofinal map in $\kappa$.
	If $(x_\alpha)_{\alpha<\kappa}$ converges to $x\in X$, then $(x_{\alpha_\xi})_{\xi<\cof(\kappa)}$ converges to $x$.
\end{proposition}

\begin{proof}
	Fix a neighborhood $U^x\subset X$ of $x$ arbitrary.
	Since $(x_\alpha)_{\alpha<\kappa}$ converges to $x$, there is an $\alpha_0<\kappa$ such that $x_\alpha \in U^x$ for any $\alpha>\alpha_0$ with $\alpha<\kappa$.
	As $f$ is a cofinal map in $\kappa$, there exists $\xi_0<\cof(\kappa)$ such that $\alpha_{\xi_0}=f(\xi_0)\geq \alpha_0$.
	Moreover, since $f$ is strictly increasing, for every $\xi>\xi_0$ with $\xi<\cof(\kappa)$ we have that $\alpha_\xi=f(\xi)>f(\xi_0)=\alpha_{\xi_0}\geq \alpha_0$.
	Hence, for any $\xi>\xi_0$ with $\xi<\cof(\kappa)$, we arrive at $x_{\alpha_\xi}\in U^x$ where $U^x$ is an arbitrary neighborhood of $x$, that is, $(x_{\alpha_\xi})_{\xi<\cof(\kappa)}$ converges to $x$.
\end{proof}

In view of Proposition~\ref{prop:1}, the convergence of $\kappa$-sequences $(x_\alpha)_{\alpha<\kappa}$ for $\kappa\geq \omega$ can be reduced to studying the convergence of the $\cof(\kappa)$-subsequence $(x_{\alpha_\xi})_{\xi<\cof(\kappa)}$, since it also converges to the same point.
Moreover Proposition~\ref{prop:1} is trivially an equivalence if $\kappa$ is regular.
However, if $\kappa$ is singular, then the reverse implication in Proposition~\ref{prop:1} may not be true even in the real line as the following result shows.

\begin{proposition}\label{prop:2}
	Assume that $\mathfrak c$ is a singular cardinal.
	If 
		\begin{align*}
		f \colon \cof(\mathfrak c) &\to \mathfrak c\\
		\xi &\mapsto f(\xi)=:\alpha_\xi
		\end{align*}
	is a strictly increasing cofinal map in $\mathfrak c$,
	then there exists a $\mathfrak c$-sequence $(x_\alpha)_{\alpha<\mathfrak c}$ in $\R$ that does not converge to $0$ such that $(x_{\alpha_\xi})_{\xi<\cof(\mathfrak c)}$ converges to $0$.
\end{proposition}

\begin{proof}
	Let $(x_{\alpha_\xi})_{\xi<\cof(\mathfrak c)}$ be any $\cof(\mathfrak c)$-sequence converging to $0$.
	Since $\cof(\mathfrak c)<\mathfrak c$ we have that $\card(\mathfrak c\setminus \cof(\mathfrak c))=\mathfrak c$.
	Hence we can enumerate $[1,2]$, without repetition, as $\{x_\alpha \colon \alpha\in \mathfrak c\setminus \cof(\mathfrak c) \}$.
	Clearly, the $\mathfrak c$-sequence $(x_\alpha)_{\alpha<\mathfrak c}$ does not converge to $0$.
\end{proof}

In view of Propositions~\ref{prop:1} and~\ref{prop:2}, we will be interested in our cases for practical purposes when the $\kappa$-sequences are such that $\kappa$ is regular.
Thus, for the rest of this paper, we shall assume that the cardinal number $\kappa$ indexing the sequences is chosen to be regular.

In the following four theorems we study the coneability of the families of $\kappa$-sequences $\mathcal{NM}_\kappa$, $\mathcal{NF}_\kappa$, $\mathcal{ND}_\kappa$ and $\mathcal{ANM}_\kappa\setminus \mathcal{ND}_{\kappa}$.
In particular, in Theorem~\ref{thm:4}, we show that $\mathcal{NM}_{\mathfrak c}$, $\mathcal{NF}_{\mathfrak c}$ and $\mathcal{ND}_{\mathfrak c}$ are positively $2^{\mathfrak c}$-coneable assuming that $\non(\N)=\mathfrak c$; while in Theorem~\ref{thm:4_1} it is proven that $\mathcal{NM}_{\add(\N)}$, $\mathcal{NF}_{\add(\N)}$ and $\mathcal{ND}_{\add(\N)}$ are positively $2^{\mathfrak c}$-coneable assuming that $\add(\N)=\cov(\N)$.
In Theorem~\ref{thm:8_1}, we show that $\mathcal{ANM}_{\mathfrak c}\setminus \mathcal{ND}_{\mathfrak c}$ is $2^{\mathfrak c}$-coneable assuming that $\non(\N)=\mathfrak c$; and in Theorem~\ref{thm:8_2}, we prove that $\mathcal{ANM}_{\add(\N)}\setminus \mathcal{ND}_{\add(\N)}$ is $2^{\mathfrak c}$-coneable assuming that $\add(\N)=\cov(\N)$.

In view of the fact that $\mathcal{AN}_{\omega_1}=\emptyset$ if, and only if, $\add(\N)>\omega_1$ (see \cite{Di}, and also \cite{Na1}), it is natural to think that the following results may not be true for the families $\mathcal{NM}_{\omega_1}$, $\mathcal{NF}_{\omega_1}$, $\mathcal{ND}_{\omega_1}$ and $\mathcal{ANM}_{\omega_1}\setminus \mathcal{ND}_{\omega_1}$.

\begin{theorem}\label{thm:4}
	If $\non(\N)=\mathfrak c$, the sets $\mathcal{NM}_{\mathfrak c}$, $\mathcal{NF}_{\mathfrak c}$ and $\mathcal{ND}_{\mathfrak c}$ in $\left(\mathbb R^{[0,1]} \right)^{\mathfrak c}$ are positively $2^{\mathfrak c}$-coneable.
\end{theorem}

\begin{proof}
	Let $\{x_\beta \colon \beta<\mathfrak c \}$ be an enumeration, without repetition, of $\left[0,\frac{1}{2} \right]$.
	Take $\mathfrak C$ a Cantor set of measure $0$ contained in  $\left(\frac{1}{2},1 \right]$.
	For each $\alpha<2^{\mathfrak c} $, let us define the $\mathfrak c$-sequence of functions $(f_\beta^\alpha)_{\beta < \mathfrak c}$ as follows: for every $\beta<\mathfrak c$,
	$$
	f_\beta^\alpha:=\chi_{V_\alpha^{\mathfrak C} \cup \{x_\gamma \colon \gamma\leq \beta \}},
	$$
	where the sets $V_\alpha^{\mathfrak C}$ are defined in \eqref{equ:1} at the beginning of Section~\ref{sec:2}.
	Observe that $\card(\{x_\gamma \colon \gamma\leq \beta \})=\card(\beta+1)<\mathfrak c=\non(\N)$ (recall that infinite cardinal numbers are limit ordinals), which implies that the set $\{x_\gamma \colon \gamma\leq \beta \}$ is a null set. 
	
	Since $\sigma_\beta:=\{x_\gamma \colon \gamma\leq \beta \}\in \N$ for every $\beta<\non(\N)$, $\sigma_\xi\subseteq \sigma_\beta$ if $\xi\leq \beta<\non(\N)$, $(\chi_{\sigma_\beta})_{\beta<\non(\N)}$ converges pointwise everywhere to $\chi_I$ where $I:=\left[0,\frac{1}{2} \right]$, and $\sigma_\beta \cap \mathfrak C=\emptyset$ for every $\beta<\non(\N)$, we have the desired result by Lemma~\ref{lem:7}.
\end{proof}

\begin{theorem}\label{thm:4_1}
	If $\add(\N)=\cov(\N)$, then $\mathcal{NM}_{\add(\N)}$, $\mathcal{NF}_{\add(\N)}$ and $\mathcal{ND}_{\add(\N)}$ in $\left(\mathbb R^{[0,1]} \right)^{\add(\N)}$ are positively $2^{\mathfrak c}$-coneable.
%	\ck{The argument for this seems very similar to one I suggested in previous remark. Perhaps, the theorems have 
%		common strengthening. 
%		\\}
%	\ch{\rv{As before, I was not able to do it.}\\}
\end{theorem}

\begin{proof}
	Let $\{N_\beta \colon \beta < \cov(\N) \}$ be a collection of null sets witnessing $\cov(\N)$ contained in $\left[0,\frac{1}{2}\right]$ whose union is $\left[0,\frac{1}{2}\right]$,
%	\ck{You cannot ensure that the union is  $\left[0,\frac{1}{2}\right]$
%	unless you make stronger assumption that $\add(\N)=\mathfrak c$.\\
%	I do not see how this construction can give your theorem under weaker assumptions.
%	(Proof of Thm 3.11 from previous version had also same problem, that I missed earlier.\\}
	and $\mathfrak C$ a Cantor set of measure $0$ contained in  $\left(\frac{1}{2},1 \right]$.
	For each $\alpha<2^{\mathfrak c} $, let us define the $\add(\N)$-sequence of functions $(f_\beta^\alpha)_{\beta < \add(\N)}$ as follows: for every $\beta < \add(\N)$,
	$$
	f_\beta^\alpha:=\chi_{V_\alpha^{\mathfrak C} \cup \left(\bigcup_{\gamma\leq \beta} N_\beta \right)},
	$$
	where the sets $V_\alpha^{\mathfrak C}$ are defined in \eqref{equ:1} at the beginning of Section~\ref{sec:2}.
	As $\add(\N)=\cov(\N)$ and $\beta<\add(\N)$ it is clear that $\bigcup_{\gamma\leq \beta} N_\beta$ is a null set.
	
	Since $\sigma_\beta:=\bigcup_{\gamma\leq \beta} N_\beta\in \N$ for every $\beta<\add(\N)$, $\sigma_\xi\subseteq \sigma_\beta$ if $\xi\leq \beta<\add(\N)$, $(\chi_{\sigma_\beta})_{\beta<\add(\N)}$ converges pointwise everywhere to $\chi_I$ where $I:=\left[0,\frac{1}{2} \right]$, and $\sigma_\beta \cap \mathfrak C=\emptyset$ for every $\beta<\add(\N)$, we have the desired result by Lemma~\ref{lem:7}.
\end{proof}

\begin{theorem}\label{thm:8_1}
	If $\non(\N)=\mathfrak c$, the set $\mathcal{ANM}_{\mathfrak c} \setminus \mathcal{ND}_{\mathfrak c}$ in $\left(\mathbb R^{[0,1]} \right)^{\mathfrak c}$ is $2^{\mathfrak c}$-coneable.
\end{theorem}

\begin{proof}
	Let $\{x_\beta \colon \beta<\mathfrak c \}$ be an enumeration, without repetition, of 
	$\left(0,\frac{1}{2} \right]$,  %\ck{Why not $\left[0,\frac{1}{2} \right]$?\\}
	and $\mathfrak C$ be a Cantor set of measure $0$ contained in  $\left(\frac{1}{2},1 \right]$.
	For every $\alpha<2^{\mathfrak c} $, define the $\mathfrak c$-sequence of functions $(f_\beta^\alpha)_{\beta < \mathfrak c}$ in the following way: for every $x\in [0,1]$ and $\beta<\mathfrak c$
	$$
	f_\beta^\alpha(x):=\begin{cases}
	\frac{1}{x} \cdot \chi_{V_\alpha^{\mathfrak C} \cup \{x_\gamma \colon \gamma\leq \beta \}}(x) & \text{if } x\neq 0, \\
	0 & \text{if } x=0,
	\end{cases}
	$$
	where the sets $V_\alpha^{\mathfrak C}$ are defined in \eqref{equ:1} at the beginning of Section~\ref{sec:2}.
	It is clear that $\{x_\gamma \colon \gamma\leq \beta \}\in \N$ from the fact that $\card(\{x_\gamma \colon \gamma\leq \beta \})=\card(\beta+1)<\mathfrak c=\non(\N)$. 
	
	By taking $\sigma_\beta:=\{x_\gamma \colon \gamma\leq \beta \}\in \N$ and $I:=\left(0,\frac{1}{2} \right]$ we have that $\sigma_\beta\subset \left(0,\frac{1}{2} \right]$ for every $\beta<\mathfrak c$ and
	$(\chi_{\sigma_\beta})_{\beta<\mathfrak c}$
	converges pointwise everywhere to $\chi_I$.
	Also
	$\sigma_\xi \subseteq \sigma_\beta$ for every $\xi \leq \beta<\mathfrak c$, and $\sigma_\beta\cap \mathfrak C=\emptyset$ for every $\beta<\mathfrak c$ where $\mathfrak C\subset \left(\frac{1}{2},1 \right]$.
	Thus, by Lemma~\ref{lem:8}, the result follows.
\end{proof}

\begin{theorem}\label{thm:8_2}
	Assuming that $\add(\N)=\cov(\N)$, the set $\mathcal{ANM}_{\add(\N)} \setminus \mathcal{ND}_{\add(\N)}$ in $\left(\mathbb R^{[0,1]} \right)^{\add(\N)}$ is $2^{\mathfrak c}$-coneable.
%	\ck{Remove this theorem.
%	As in case of Thm 3.4, the proof works only under assumption that $\add(\N)=\mathfrak c$.
%	But such result follows from thm 3.5.\\}
	\end{theorem}

\begin{proof}
	Let $\{N_\beta \colon \beta < \cov(\N) \}$ be a collection of null sets witnessing $\cov(\N)$ contained in $\left(0,\frac{1}{2}\right]$ whose union is $\left(0,\frac{1}{2}\right]$, and $\mathfrak C$ a Cantor set of measure $0$ contained in  $\left(\frac{1}{2},1 \right]$.
	For every $\alpha<2^{\mathfrak c} $, we will define the $\add(\N)$-sequence of functions $(f_\beta^\alpha)_{\beta < \mathfrak c}$ in the following way: for every $x\in [0,1]$ and $\beta<\add(\N)$,,
	$$
	f_\beta^\alpha(x):=\begin{cases}
	\frac{1}{x} \cdot \chi_{V_\alpha^{\mathfrak C} \cup \left(\bigcup_{\gamma\leq \beta} N_\beta \right)}(x) & \text{if } x\neq 0, \\
	0 & \text{if } x=0,
	\end{cases}
	$$
	where the sets $V_\alpha^{\mathfrak C}$ are defined in \eqref{equ:1} at the beginning of Section~\ref{sec:2}.
	Since $\add(\N)=\cov(\N)$ and $\beta<\add(\N)$ note that $\bigcup_{\gamma\leq \beta} N_\beta\in N$. 
	
	Since $\sigma_\beta:=\bigcup_{\gamma\leq \beta} N_\beta\in \N$ satisfies that $\sigma_\beta\subset \left(0,\frac{1}{2} \right]$ for each $\beta<\mathfrak c$, $(\chi_\beta)_{\beta<\mathfrak c}$ converges pointwise everywhere to $\chi_I$ with $I:=\left(0,\frac{1}{2} \right]$, $\sigma_\xi\subseteq \sigma_\beta$ for every $\xi \leq \beta<\mathfrak c$, and $\sigma_\beta\cap \mathfrak C=\emptyset$ for any $\beta<\mathfrak c$ where $\mathfrak C\subset \left[\frac{1}{2},1 \right]$, we have the desired result by Lemma~\ref{lem:8}.
\end{proof}

As in Section~2, by comparing Theorem~\ref{thm:7} with Theorem~\ref{thm:8}, and Theorem~\ref{thm:1} with Thoerem~\ref{thm:3_1}, we can improve the algebraic structure of: (1) $\mathcal{ND}_{\mathfrak c}$ and $\mathcal{ANM}_{\mathfrak c}\setminus \mathcal{ND}_{\mathfrak c}$ from Theorems~\ref{thm:4} and~\ref{thm:8_1}, respectively (under $\non(\N)=\mathfrak c$), and (2) $\mathcal{ND}_{\add(\N)}$ and $\mathcal{ANM}_{\add(\N)}\setminus \mathcal{ND}_{\add(\N)}$ from Theorems~\ref{thm:4_1} and~\ref{thm:8_2}, respectively (under $\add(\N)=\cov(\N)$).
This is shown in Theorems~\ref{thm:9_1}, \ref{thm:9_2}, \ref{thm:6} and~\ref{thm:6_2} below.

\begin{theorem}\label{thm:9_1}
	If $\non(\N)=\mathfrak c$, then we have that $\mathcal{ANM}_{\mathfrak c}\setminus \mathcal{ND}_{\mathfrak c}$ in $\left(\mathbb R^{[0,1]}\right)^{\mathfrak c}$ is strongly $\mathfrak c$-algebrable.
\end{theorem}

\begin{proof}
	Let $\mathcal H$ be a Hamel basis of $\R$ over $\mathbb Q$ and let $\{x_\alpha \colon \alpha<\mathfrak c \}$ be an enumeration, without repetition, of $\left[\frac{2}{3},1 \right]$.
	For every $h\in \mathcal H$, define the $\mathfrak c$-sequences $(f_\alpha^h)_{\alpha<\mathfrak c}$ as follows: for each $\alpha<\mathfrak c$,
	$$
	f_\alpha^h(x):=\begin{cases}
	f_h(x) \cdot \chi_{\left(0,\frac{1}{2}\right] \cup \{x_\beta \colon \beta\leq \alpha \}} & \text{if } x\neq 0,\\
	0 & \text{if } x=0,
	\end{cases}
	$$
	where $f_h(x):=\frac{e^{hx}}{x}$ for every $x\in [0,1]$.
	Note that $\{x_\beta \colon \beta\leq \alpha \}\in \N$ since $\alpha<\mathfrak c$ and $\non(\N)=\mathfrak c$.
	
	By taking $\sigma_\alpha:=\{x_\beta \colon \beta\leq \alpha \}\in \N$ for each $\alpha<\mathfrak c$, we have that $\sigma_\alpha \cap I=\emptyset$ for each $\alpha<\mathfrak c$ with $I:=\left(0,\frac{1}{2}\right]$, $\sigma_\alpha\subseteq \sigma_\beta$ for every $\alpha\leq \beta<\mathfrak c$, and $(\chi_{\sigma_\alpha})_{\alpha<\mathfrak c}$ converges pointwise everywhere to $\chi_J$ with $J:=\left[\frac{2}{3},1 \right]$.
	Thus, by Lemma~\ref{lem:9}, we have the desired result.
\end{proof}

\begin{theorem}\label{thm:9_2}
	If $\add(\N)=\cov(\N)$, we have that $\mathcal{ANM}_{\add(\N)}\setminus \mathcal{ND}_{\add(\N)}$ in $\left(\mathbb R^{[0,1]}\right)^{\add(\N)}$ is strongly $\mathfrak c$-algebrable.
\end{theorem}

\begin{proof}
	Let $\mathcal H$ be a Hamel basis of $\R$ over $\mathbb Q$ 
	and take $\{N_\alpha \colon \alpha < \cov(\N) \}$ a collection of null sets witnessing $\cov(\N)$ contained in $\left[\frac{2}{3},1\right]$ such that $\bigcup_{\alpha<\cov(\N)} N_\alpha=\left[\frac{2}{3},1\right]$.
%	\kc{???}\ck{Same problem as earlier. 
%	The proof works only under assumption that $\add(\N)=\mathfrak c$.\\}
	For each $h\in \mathcal H$, let us define the $\add(\N)$-sequences $(f_\alpha^h)_{\alpha<\add(\N)}$ as follows: for each $\alpha<\ad(\N)$,
	$$
	f_\alpha^h(x):=\begin{cases}
	f_h(x) \cdot \chi_{\left(0,\frac{1}{2}\right] \cup \left(\bigcup_{\beta\leq \alpha} N_\beta\right)} & \text{if } x\neq 0,\\
	0 & \text{if } x=0,
	\end{cases}
	$$
	where $f_h(x):=\frac{e^{hx}}{x}$ for every $x\in [0,1]$.
	If $\alpha<\add(\N)=\cov(\N)$, observe that $\bigcup_{\beta\leq \alpha} N_\beta\in \N$.
	
	By taking $\sigma_\alpha:=\bigcup_{\beta\leq \alpha} N_\beta\in \N$ for each $\alpha<\mathfrak c$, we have that $\sigma_\alpha \cap I=\emptyset$ for each $\alpha<\mathfrak c$ with $I:=\left(0,\frac{1}{2}\right]$, $\sigma_\alpha\subseteq \sigma_\beta$ for every $\alpha\leq \beta<\add(\N)$, and $(\chi_{\sigma_\alpha})_{\alpha<\mathfrak c}$ converges pointwise everywhere to $\chi_J$ with $J:=\left[\frac{2}{3},1 \right]$.
	Therefore, the result follows from Lemma~\ref{lem:9}.
	\end{proof}

\begin{theorem}\label{thm:6}
	If $\non(\mathcal N)=\mathfrak c$, then $\mathcal{ND}_{\mathfrak c}$ in $\left(\mathbb R^{[0,1]}\right)^{\mathfrak c}$ is strongly $\mathfrak c$-algebrable.
\end{theorem}

\begin{proof}
	Let $\mathcal H$ be a Hamel basis of $\mathbb R$ over $\mathbb Q$ and let $\{x_\alpha \colon \alpha<\mathfrak c \}$ be an enumeration, without repetition, of $\left[0,\frac{1}{2}\right]$.
	For every $h\in \mathcal H$, let us define a $\mathfrak c$-sequence of functions $(f_\alpha^h)_{\alpha<\mathfrak c}$ in the following way: for each $\alpha<\mathfrak c$,
	$$
	f_\alpha^h=f_h\cdot \chi_{\{x_\beta \colon \beta\leq \alpha \}\cup \left(\frac{1}{2},1\right]},
	$$
	where $f_h(x)=e^{hx}$ for every $x\in [0,1]$.
	As $\card(\{x_\beta \colon \beta\leq \alpha \})=\card(\alpha+1)<\mathfrak c$ and $\non(\mathcal N)=\mathfrak c$, we have that $\{x_\beta \colon \beta\leq \alpha \}$ is a null set.
	
	Observe that $\sigma_\alpha:=\{x_\beta \colon \beta\leq \alpha \}\in \N$ for every $\alpha<\mathfrak c$, $\sigma_\alpha\subset J\setminus I$ with $J=\left[0,\frac{1}{2}\right]$ and $I=\left(\frac{1}{2},1\right]$ for every $\alpha<\mathfrak c$, and $(\chi_{\sigma_\alpha})_{\alpha<\add(\N)}$ converges pointwise everywhere to $\chi_J$.
	Therefore, by the first part and the additional part (ii) of Lemma~\ref{lem:5}, we have the desired 
	result.
\end{proof}

\begin{theorem}\label{thm:6_2}
	If $\add(\N)=\cov(\N)$, we have that $\mathcal{ND}_{\add(\N)}$ in $\left(\mathbb R^{[0,1]}\right)^{\add(\N)}$ is strongly 
	$\mathfrak c$-algebrable.
\end{theorem}

\begin{proof}
	Let $\mathcal H$ be a Hamel basis of $\mathbb R$ over $\mathbb Q$ and take $\{N_\alpha \colon \alpha < \cov(\N) \}$ a collection of null sets witnessing $\cov(\N)$ contained in $\left[0,\frac{1}{2}\right]$ such that $\bigcup_{\alpha<\cov(\N)} N_\alpha=\left[0,\frac{1}{2}\right]$.
%	\kc{???}\ck{Remove this theorem.
%	As in case of Thm 3.4, the proof works only under assumption that $\add(\N)=\mathfrak c$.
%	But such result follows from thm 3.9.\\}
	%	\ch{Remarks as above, it seems to me that using $\kappa=\non(\mathcal N)$, you can prove in ZFC your Theorem 3.3.\\} 
	For every $h\in \mathcal H$, let us define the $\add(\N)$-sequences of functions $(f_{\alpha}^h)_{\alpha<\add(\N)}$ in the following way: for each $\alpha<\add(\N)$,
	$$
	f_\alpha^h=f_h\cdot \chi_{\left(\frac{1}{2},1\right] \cup \left(\bigcup_{\beta\leq \alpha} N_\beta\right)},
	$$
	where $f_h(x)=e^{hx}$ for every $x\in [0,1]$.
	Since $\add(\N)=\cov(\N)$ and $\alpha+1<\add(\N)$ we have that the union $\bigcup_{\beta\leq \alpha} N_\beta$ is well defined and a null set.
	
	Note that $\sigma_\alpha:=\bigcup_{\beta\leq \alpha} N_\beta\in \N$ for every $\alpha<\add(\N)$, $\sigma_\alpha\subset J\setminus I$ with $J=\left[0,\frac{1}{2}\right]$ and $I=\left(\frac{1}{2},1\right]$, for every $\alpha<\add(\N)$ and $(\chi_{\sigma_\alpha})_{\alpha<\add(\N)}$ converges pointwise everywhere to $\chi_J$.
	Hence, by the first part and the additional part (ii) of Lemma~\ref{lem:5}, we have the desired 
	result.
\end{proof}

Recall that if the CH fails, it may not be true that there exists an algebraic structure contained in $\mathcal{AN}_{\omega_1}$, since $\mathcal{AN}_{\omega_1}=\emptyset$ if, and only if, $\add(\mathcal N)>\omega_1$.
Nonetheless, the next result shows that by considering sequences of size $\add(\N)$ we can obtain an algebraic structure contained in $\mathcal{AN}_{\add(\N)}$.
In particular, we prove that $\mathcal{AN}_{\add(\N)}$ in $\left(\mathbb R^\R \right)^{\add(\N)}$ is strongly $\mathfrak c$-algebrable.

\begin{theorem}\label{thm:6_1}
	The set $\mathcal{AN}_{\add(\N)}$ in $\left(\mathbb R^\R \right)^{\add(\N)}$ is strongly $\mathfrak c$-algebrable.%
\end{theorem}

\begin{proof}
	Let $\mathcal H$ be a Hamel basis of $\R$ over $\mathbb Q$ and take, for every $h\in \mathcal H$,
		$$
		f_h(x):=e^{h x},
		$$
	for any $x\in \R$.
	
	Let $\mathcal A$ be a family of pairwise disjoint null sets such that $\card(\mathcal A)=\add(\N)$, $\bigcup \mathcal A\notin \N$ and each set of $\mathcal A$ is contained in $\left[0,1\right]$.
	By Lemma~\ref{lem:2}, take $V\subseteq \bigcup \mathcal A$ a nonmeasurable set of inner measure $0$ and same outer measure as $\bigcup \mathcal A$.
	Now let $\{N_\alpha \colon \alpha<\add(\N) \}$ be an enumeration, without repetition, of $\mathcal A$.
	For every $\beta<\add(\N)$, take $N_\beta^\prime :=N_\beta \cap V$ and let $\overline{\mathcal N}:=\{N_\beta^\prime \colon \beta<\add(\N) \}\setminus \{\emptyset\}$. (Note that there might exist $\beta<\add(\N)$ such that $N_\beta^\prime=\emptyset$.)
	Clearly, by construction, we have that $V=\bigcup \overline{\mathcal N}$.
	Moreover, $\text{card}(\overline{\mathcal N})=\add(\N)$.
	Indeed, it is obvious that $\text{card}(\overline{\mathcal N})\leq \add(\N)$, so assume that $\text{card}(\overline{\mathcal N})< \add(\N)$.
	Then, by definition of $\add(\N)$, the set $V$ would be a null set, a contradiction.

	Let $\{\overline{N}_\beta \colon \beta<\add(\N) \}$ be an enumeration, without repetition, of $\overline{\mathcal N}$.
	For every $h\in \mathcal H$, let us define an $\add(\N)$-sequence of functions $(f_\alpha^h)_{\alpha<\add(\N)}$ in the following way: for each $\alpha<\add(\N)$, 
		$$
		f_{\alpha}^h:=f_h\cdot \chi_{\left(\bigcup_{\beta \leq \alpha} \overline{N}_\beta \right) \cup (1,2)}.
		$$	
%	\ck{What is $\bigcup_{\beta<\alpha^+} N_\beta^\prime$? 
%	Do you mean $\bigcup_{\beta\leq\alpha} N_\beta^\prime$?\\
%	The symbol $\alpha^+$ is used only when $\alpha$ is cardinal, not for general ordinal numbers.
%	Ordinal successor is denoted $\alpha+1$.\\}
%	\ch{\rv{Yes, you are right. The notation from set theory is a little fuzzy in my mind.}\\}
	As $\card(\alpha+1)<\add(\N)$ we have that $\bigcup_{\beta \leq \alpha} \overline{N}_\beta $ is a null set.
	
	Now, since $\sigma_\alpha:=\bigcup_{\beta \leq \alpha} \overline{N}_\beta \in \N$ for every $\alpha<\add(\N)$, $\sigma_\alpha \subset J\setminus I$ with $J=[0,1]$ and $I=(1,2)$ for every $\alpha<\add(\N)$, and $(\chi_{\sigma_\alpha})_{\alpha<\add(\N)}$ converges pointwise everywhere to $\chi_V$ where $V\subset J\setminus I$ is a nonmeasurable set with inner measure $0$, we have the desired result by the first part and the additional part (i) of Lemma~\ref{lem:5}. 
%	\ck{Some sets $N_\beta^\prime$ can be empty. 
%	Doesn't it influence argument for algebraic independence?\\}
%	\ch{\rv{It kind of does but it can be fixed.}\\}
\end{proof}

\begin{corollary}\label{cor:3}
	$\mathrm{LIM}_{\add(\N)}(\mathcal L)\supsetneq \mathcal L$.
\end{corollary}

\begin{corollary}\label{cor:4}
	It is independent of $\ZFC$ that $\mathrm{LIM}_{\add(\N)}(\mathcal L)=\mathrm{LIM}_{\omega_1}(\mathcal L)$.
\end{corollary}

\begin{proof}
	Under CH it is clear that $\text{LIM}_{\add(\N)}(\mathcal L)=\text{LIM}_{\omega_1}(\mathcal L)$.
	
	Under MA+$\neg$CH, in which $\omega_1<\add(\N)$ (\cite[theorem~2.21]{Ku}), we have that $\text{LIM}_{\omega_1}(\mathcal L)=\mathcal L$ by the proof of \cite[theorem~2]{Di} (see also \cite[remark~2]{Di}), while $\text{LIM}_{\add(\N)}(\mathcal L)\neq \mathcal L$ from Corollary~\ref{cor:3}.
\end{proof}

By considering $\kappa$-sequences in $\mathcal{AN}_{\kappa}$ with an appropiate $\kappa\geq \add(\N)$, we may improve the size of the generators of an algebra contained in $\mathcal{AN}_{\kappa}$ in comparison with Theorem~\ref{thm:6_1}.
However, the property of strong algebrability may be lost.
This is shown in Theorems~\ref{thm:6_1_3}, \ref{thm:6_1_1} and~\ref{thm:6_1_2} below.

%\ck{Argument in blue is well too complicated. See my argument.\\} 
%\ch{\rv{Yes, it is simpler. I put this argument instead.}\\}

\begin{theorem}\label{thm:6_1_3}
	The set $\mathcal{AN}_{\non(\N)}$ in $\left(\mathbb R^\R \right)^{\non(\N)}$ is $2^{\non(\N)}$-algebrable.
\end{theorem}

\begin{proof}
	If $\non(\N)=\mathfrak c$, take $\{Z_\xi \colon \xi<\mathfrak c \}$ as a family of $\mathfrak c$-many pairwise disjoint Bernstein sets.
	If $\non(\N)<\mathfrak c$, take $Z\subset \R$ a nonmeasurable set with $\card(Z)=\non(\N)$ and put 
		$Z_\xi:=Z+r_\xi$, where numbers $r_\xi\in \R$ are chosen inductively so that 
		$$
		(Z+r_\xi)\cap \left(\bigcup_{\eta<\xi} (Z+r_\eta)\right)=\emptyset,
		$$
	i.e., $r_\xi\in\R\setminus \bigcup_{\eta<\xi} \left(r_\eta+(Z-Z)\right)$.
	This can be done since 
		$$
		\card\left(\bigcup_{\eta<\xi} \left(r_\eta+(Z-Z)\right) \right)<\mathfrak c.
		$$
	Note that, in the case when $\non(\N)<\mathfrak c$, the family $\{Z_\xi \colon \xi<\mathfrak c \}$ is a family of $\mathfrak c$-many pairwise disjoint nonmeasurable subsets of $\R$ each one having cardinality $\non(\N)$.
	
	By the Fichtenholz-Kantorovich-Hausdorff theorem, let $\{D_\alpha \colon \alpha< 2^{\non(\N)} \}$ be an independent family of subsets of $\non(\N)$ of cardinality $2^{\non(\N)}$ such that each $D_\alpha$ has cardinality $\non(\N)$.
	For every $\alpha<2^{\non(\N)}$, let
		$$
		V_\alpha := \bigcup_{\xi\in D_\alpha} Z_\xi.
		$$
	Notice that $\card(V_\alpha)=\non(\N)$.
	
	For each $\alpha<2^{\non(\N)}$, let $\{x_{\alpha,\beta} \colon \beta<\non(\N) \}$ be an enumeration, without repetition, of $V_\alpha$; and let us define the $\non(\N)$-sequence of functions $(f_\beta^\alpha)_{\beta<\non(\N)}$ as  follows: for every $\beta<\non(\N)$,
		$$
		f_\beta^\alpha=\chi_{\{x_{\alpha,\gamma} \colon \gamma\leq \beta \}}.
		$$
	Since $\card(\{x_{\alpha,\gamma} \colon \gamma \leq \beta \})=\card(\beta+1)<\non(\N)$, it is obvious that $\{x_{\alpha,\gamma} \colon \gamma\leq \beta \}$ is a null set for every $\alpha<2^{\non(\N)}$. 
	Therefore, each $f_\beta^\alpha$ is measurable.
	
	It remains to prove that the $\non(\N)$-sequences of functions in 
		$$
		\mathcal S:=\{(f_\beta^\alpha)_{\beta<\non(\N)} \colon \alpha<2^{\non(\N)} \}
		$$ 
	is linearly independent, $\mathcal S$ is a minimal system of generators of an algebra and any nonzero algebraic combination of the $\non(\N)$-sequences in $\mathcal S$ belongs to $\mathcal{AN}_{\non(\N)}$.

	Let $\lambda_1,\ldots,\lambda_n\in \mathbb R$, $\alpha_1,\ldots,\alpha_n<2^{\non(\N)}$ be distinct and take the $\non(\N)$-sequence of functions
		$$
		(F_\beta)_{\beta<\non(\N)}:=\sum_{i=1}^n \lambda_i (f_\beta^{\alpha_i})_{\beta<\non(\N)}.
		$$
	Assume that $(F_\beta)_{\beta<\non(\N)}\equiv 0$.
	By definition of a family of independent subsets, there is a $\xi\in D_{\alpha_1}\setminus \bigcup_{i=2}^n D_{\alpha_i}$.
	Hence, by construction, $Z_\xi \subset V_{\alpha_1}\setminus \bigcup_{i=2} V_{\alpha_i}$.
	Now take any $x\in Z_\xi$, then there exists $\beta<\non(\N)$ such that $x=x_{\alpha_1,\beta}$.
	Moreover, observe that $f_\beta^{\alpha_1}(x)=1$ and $f_\beta^{\alpha_i}(x)=0$ for every $i\in \{2,\ldots,n \}$.
	Thus, we have
		$$
		F_\beta(x)=\lambda_1=0.
		$$
	By repeating similar arguments, note that $\lambda_i=0$ for every $i\in \{1,\ldots,n \}$.
	
	Now let us prove that $\mathcal S$ is a minimal system of generators of the algebra generated by $\mathcal S$.
	Take $\alpha<2^{\non(\N)}$ and assume that $(f_\beta^\alpha)_{\beta<\non(\N)}$ belongs to the algebra generated by $\mathcal S\setminus \{(f_\beta^\alpha)_{\beta<\non(\N)} \}$.
	Then there exist $\alpha_1,\ldots,\alpha_m<2^{\non(\N)}$ distinct (with $m\in \mathbb N$) and $\lambda_1,\ldots,\lambda_n\in \R\setminus \{0\}$ (with $n\in \mathbb N$) such that
		$$
		(f_\beta^\alpha)_{\beta<\non(\N)}=\sum_{i=1}^n \lambda_i \prod_{j=1}^m (f_\beta^{\alpha_j})_{\beta<\non(\N)}^{k_{j,i}}
		$$
	with $k_{j,i}\in \mathbb N_0$, $\sum_{j=1}^m k_{j,i}\geq 1$ and the $m$-tuples $(k_{1,i},\ldots,k_{m,i})$ are distinct for every $i\in \{1,\ldots,n \}$.
	By definition of a family of independent subsets, there exists $\xi \in D_{\alpha}\setminus \bigcup_{i=1}^m D_{\alpha_i}$.
	Now take any $x\in Z_\xi\subset V_\alpha\setminus \bigcup_{i=1}^m V_{\alpha_i}$.
	Then, by construction, there exists $\beta<\non(\N)$ such that $x=x_{\alpha,\beta}$ in which case we have that
		$$
		1=f_\beta^\alpha(x)=\sum_{i=1}^n \lambda_i \prod_{j=1}^m (f_\beta^{\alpha_j})^{k_{j,i}}(x)=0,
		$$
	a contradiction.
	
	Let $\alpha_1,\ldots,\alpha_m<2^{\non(\N)}$ be distinct, $\lambda_1,\ldots,\lambda_n\in \R\setminus \{0\}$ and 
		$$
		(F_\beta)_{\beta<\non(\N)}:=\sum_{i=1}^n \lambda_i \prod_{j=1}^m (f_\beta^{\alpha_j})_{\beta<\non(\N)}^{k_{j,i}}
		$$
	with $k_{j,i}\in \mathbb N_0$, $\sum_{j=1}^m k_{j,i}\geq 1$ and the $m$-tuples $(k_{1,i},\ldots,k_{m,i})$ are distinct for every $i\in \{1,\ldots,n \}$.
	Assume that $(F_\beta)_{\beta<\non(\N)}\not\equiv 0$.
	It is obvious that $(F_\beta)_{\beta<\non(\N)}$ is a $\non(\N)$-sequence of measurable functions that converges pointwise everywhere to
		$$
		F:=\sum_{i=1}^n \lambda_i \prod_{j=1}^m \chi_{V_{\alpha_j}}^{k_{j,i}}.
		$$
	Let us show that $F$ is nonmeasurable.
	Since $(F_\beta)_{\beta<\non(\N)}$ is not identically $0$, there exist $\beta_0<\non(\N)$ and $\{i_1,\ldots,i_r \}\subset \{1,\ldots,n \}$ (with $r\in \{1,\ldots,n \})$ such that $F_\beta[\R]\supset \left\{\sum_{j=1}^r \lambda_{i_j} \right\}$ for every $\beta_0\leq \beta<\non(\N)$.
	Let us denote $\sum_{j=1}^r \lambda_{i_j}$ by $\lambda$.
	Note that, by construction, we have $F[\R]\supset \{\lambda \}$.
	Moreover, notice that $F^{-1}(\lambda)$ is the finite union of sets of the form
		$$
		\Lambda:=\bigcup_{\xi \in \bigcap_{i=1}^n D_{\alpha_i}^{\varepsilon_i}} Z_\xi,
		$$
	where $\varepsilon_i\in \{0,1 \}$ for every $i\in \{1,\ldots,n \}$ and $\sum_{i=1}^n \varepsilon_i\geq 1$.
	(Recall that $Y^1:=Y$ and $Y^0:=X\setminus Y$ for any sets $Y\subseteq X$, see Definition~\ref{def:1}.)
	On the one hand, if $\non(\N)=\mathfrak c$, then $F^{-1}(\lambda)$ is a Bernstein set.
	Indeed, if we take $\xi_0\in \non(\N)\setminus \bigcup_{i=1}^n D_{\alpha_i}$ we have that $Z_{\xi_0}\subset F^{-1}(0)$, i.e., $Z_{\xi_0}\cap Z_\xi=\emptyset$ for any $Z_{\xi}\subset F^{-1}(\lambda)$.
	Since $Z_{\xi_0}$ is a Bernstein set, it is clear that $F^{-1}(\lambda)$ is also a Bernstein set.
	On the other hand, if $\non(\N)<\mathfrak c$ and we assume by means of contradiction that $F^{-1}(\lambda)$ is measurable, we would have that $F^{-1}(\lambda) \in \N$ since $\card\left(\Lambda \right)=\non(\N)<\mathfrak c$.
	But the latter implies that each $Z_\xi\subset F^{-1}(\lambda)$ is a null set, 
	a contradiction.
\end{proof}

\begin{corollary}\label{cor:1}
	$\mathrm{LIM}_{\non(\N)}(\mathcal L)\supsetneq \mathcal L$.
\end{corollary}

\begin{corollary}\label{cor:2}
	It is independent of $\ZFC$ that $\mathrm{LIM}_{\non(\N)}(\mathcal L)=\mathrm{LIM}_{\omega_1}(\mathcal L)$.
\end{corollary}

\begin{proof}
	Under CH we have $\text{LIM}_{\non(\N)}(\mathcal L)=\text{LIM}_{\omega_1}(\mathcal L)$.
	
	Under MA+$\neg$CH we have $\omega_1<\add(\N)=\non(\mathcal N)$ (\cite[theorem~2.21]{Ku}).
	Hence $\text{LIM}_{\omega_1}(\mathcal L)=\mathcal L$, while $\text{LIM}_{\non(\N)}(\mathcal L)\neq \mathcal L$ from Corollary~\ref{cor:1}.
\end{proof}

For the final
two theorems, Theorems~\ref{thm:6_1_1} and~\ref{thm:6_1_2}, we have that $\mathcal{AN}_{\mathfrak c}$ and $\mathcal{AN}_{\add(\N)}$ are strongly $2^{\mathfrak c}$-algebrable under $\non(\N)=\mathfrak c$ and $\add(\N)=\cov(\N)$, respectively.

\begin{theorem}\label{thm:6_1_1}
	If $\non(\N)=\mathfrak c$, then $\mathcal{AN}_{\mathfrak c}$ in $\left(\mathbb R^\R \right)^{\mathfrak c}$ is strongly $2^{\mathfrak c}$-algebrable.
\end{theorem}

\begin{proof}
	Let $\psi\colon \beta \mathbb N\to 2^{\mathfrak c}$ be a bijection.
	For each $\mathcal U\in \beta \mathbb N$, let $\{x_{\mathcal U,\beta} \colon \beta<\mathfrak c \}$ be an enumeration, without repetition, of $V_{\psi(\mathcal U)}^\R$ (where the sets $V_{\psi(\mathcal U)}^\R$ are defined in \eqref{equ:1} at the beginning of Section~\ref{sec:2}); and consider the $\mathfrak c$-sequence of functions $(f_\beta^{\mathcal U})_{\beta< \mathfrak c}$ defined as follows: for each $\beta<\mathfrak c$,
		$$
		f_\beta^{\mathcal U}:=(g_{\mathcal U} \circ \varphi) \cdot \chi_{\{x_{\mathcal U,\gamma} \colon \gamma\leq \beta \}},
		$$
	where the functions $g_{\mathcal U}$ satisfy Theorem~\ref{CiRoSe}, and $\varphi$ is defined in the paragraph before Lemma~\ref{lem:1}.
	Since $\card(\{x_{\mathcal U,\gamma} \colon \gamma\leq \beta \})=\card(\beta+1)<\mathfrak c$ and $\non(\mathcal N)=\mathfrak c$, it is obvious that $\{x_{\mathcal U,\gamma} \colon \gamma\leq \beta \}$ is a null set for every $\mathcal U\in \beta \mathbb N$.
	
	Since $\sigma_\beta^{\mathcal U}:=\{x_{\mathcal U,\gamma} \colon \gamma\leq \beta \}\in \N$ for every $\beta<\mathfrak c$ and $\mathcal U\in \beta \mathbb N$, and $(\chi_{\sigma_\beta^{\mathcal U}})_{\beta<\mathfrak c}$ converges pointwise everywhere to $\chi_{V_{\psi(\mathcal U)}^\R}$ for every $\mathcal U\in \beta \mathbb N$, we have the desired result by Lemma~\ref{lem:1}.
\end{proof}

\begin{theorem}\label{thm:6_1_2}
	If $\add(\N)=\cov(\mathcal N)$, then $\mathcal{AN}_{\add(\N)}$ in $\left(\mathbb R^\R \right)^{\add(\N)}$ is strongly 
	$2^{\mathfrak c}$-algebrable.
\end{theorem}

\begin{proof}
	Let $\psi\colon \beta \mathbb N\to 2^{\mathfrak c}$ be a bijection.
	Also let $\{N_\beta \colon \beta<\cov(\N) \}$ be a collection of null sets such 
	that $\R=\bigcup_{\beta<\cov(\N)} N_\beta$.
	For every $\mathcal U\in \beta \mathbb N$, let us define the following $\add(\N)$-sequences: for each $\beta<\add(\N)$,
		$$
		f_\beta^{\mathcal U}:=(g_{\mathcal U}\circ \varphi) \cdot \chi_{\bigcup_{\gamma\leq \beta} \left(N_\beta\cap V_{\psi(\mathcal U)}^\R\right)},
		$$
	where the sets $V_{\psi(\mathcal U)}^\R$ were defined in \eqref{equ:1} at the beginning of Section~\ref{sec:2}, the functions $g_{\mathcal U}$ satisfy Theorem~\ref{CiRoSe}, and $\varphi$ was defined in the paragraph prior to Lemma~\ref{lem:1}.
	Notice that the unions $\bigcup_{\gamma\leq \beta} \left(N_\beta\cap V_{\psi(\mathcal U)}^\R\right) \subset \bigcup_{\gamma\leq \beta} N_\beta$ are well defined and also null sets for every $\beta<\add(\N)$ and $\mathcal U\in \beta \mathbb N$ since $\text{card}(\beta+1)<\add(\N)=\cov(\N)$.

	By taking $\sigma_\beta^\mathcal U:=\bigcup_{\gamma\leq \beta} \left(N_\beta\cap V_{\psi(\mathcal U)}^\R\right) \in \mathcal N$ for every $\beta<\add(\N)$ and $\mathcal U\in \beta \mathbb N$, we have that the $\add(\N)$-sequence $\left(\chi_{\sigma_\beta^{\mathcal U}}\right)_{\beta<\add(\N)}$ converges pointwise everywhere to $\chi_{V_{\psi(\mathcal U)}^\R}$.
	Hence, we have the desired result by Lemma~\ref{lem:1}.
\end{proof}

\section{Final remarks}\label{sec:4}

In Section~\ref{sec:2} we were interested in finding any directed set (that is not a cardinal number) in which there was an algebraic structure such as a cone or an algebra.
In any case, not only we tried to find the most optimal structure (for instance, in the case of algebras, we were trying to find algebraically independent nets) but we were also studying the largest possible dimension of these structures.
Moreover, notice that we tried to find the smallest possible cardinality of a directed set such that the family of nets in this paper contained a certain algebraic structure.
Keeping the above in mind it is natural to ask the following questions:
	\begin{itemize}
		\item[(Q1)] Let $1\leq \mu\leq 2^{\mathfrak c}$ be a cardinal number. 
		For which directed sets $\mathcal A_\mu$ does there exist a positive cone of dimension $\mu$ contained in $\mathcal{NM}_{\mathcal A}\cup \{0\}$, $\mathcal{NF}_{\mathcal A}\cup \{0\}$ and $\mathcal{ND}_{\mathcal A}\cup \{0\}$ in $\left(\mathbb R^{[0,1]} \right)^{\mathcal A}$?
		What is the smallest possible cardinality of $\mathcal A_\mu$?
		
		\item[(Q2)] Let $1\leq \mu\leq 2^{\mathfrak c}$ be a cardinal number. 
		For which directed sets $\mathcal A_\mu$ does there exist a cone of dimension $\mu$ contained in $\mathcal{ND}_{\mathcal A}\cup \{0\}$ in $\left(\mathbb R^{[0,1]} \right)^{\mathcal A}$?
		What is the smallest possible cardinality of $\mathcal A_\mu$?
		
		\item[(Q3)] Let $1\leq \mu\leq 2^{\mathfrak c}$ be a cardinal number. 
		For which directed sets $\mathcal A_\mu$ are the families $\mathcal{AN}_{\mathcal A}$ (in $\left(\R^\R \right)^{\mathcal A}$), $\mathcal{ND}_{\mathcal A}$ and $\mathcal{ANM}_{\mathcal A} \setminus \mathcal{ND}_{\mathcal A}$ (in $\left(\mathbb R^{[0,1]} \right)^{\mathcal A}$) strongly $\mu$-algebrable?
		What is the smallest possible cardinality of $\mathcal A_\mu$?
	\end{itemize}

In the case of $\kappa$-sequences (Section~\ref{sec:3}) we have also considered directed sets but with the additional restriction that they were (regular) cardinal numbers.
In some cases we have noticed that it was impossible to find any kind of algebraic structure (this is the case of $\mathcal{AN}_{\omega_1}$ assuming that $\add(\N)>\omega_1$).
And also, unlike the general case studied in Section~\ref{sec:2}, we had to consider in some cases additional set-theoretical assumptions.
In view of these results we have the following questions regarding $\kappa$-sequences:
	\begin{itemize}
		\item[(Q4)] Let $1\leq \mu\leq 2^{\mathfrak c}$ be a cardinal number. 
		Is it possible to find in ZFC a (regular) cardinal number $\kappa$ such that $\mathcal{NM}_\kappa$, $\mathcal{NF}_\kappa$ and $\mathcal{ND}_\kappa$ in $\left(\R^{[0,1]} \right)^\kappa$ are positively $\mu$-coneable? Are additional axioms needed? 
		In any case, what is the smallest possible $\kappa$?
		
		\item[(Q5)] Let $1\leq \mu\leq 2^{\mathfrak c}$ be a cardinal number. 
		Is it possible to find in ZFC a (regular) cardinal number $\kappa$ such that $\mathcal{ANM}_\kappa \setminus \mathcal{ND}_\kappa$ in $\left(\R^{[0,1]} \right)^\kappa$ is $\mu$-coneable? 
		Are additional axioms needed? 
		In any case, what is the smallest possible $\kappa$?
		
		\item[(Q6)] Let $1\leq \mu\leq 2^{\mathfrak c}$ be a cardinal number. 
		Is it possible to find in ZFC a (regular) cardinal number $\kappa$ such that $\mathcal{AN}_\kappa$ (in $\left(\R^\R \right)^\kappa$) and $\mathcal{ND}_\kappa$ ($\left(\R^{[0,1]} \right)^\kappa$) is strongly $\mu$-algebrable? 
		Are additional axioms needed? 
		In any case, what is the smallest possible $\kappa$?
	\end{itemize}

\end{document}